\documentclass{amsart}
\usepackage{amssymb}
\usepackage[mathcal]{euscript}
\usepackage[cmtip,all]{xy}
\usepackage{amsmath}\usepackage{amsfonts}
\usepackage{graphicx} 
\usepackage{verbatim} 
\usepackage{tikz}
\usepackage{hyperref} 

    \newcommand{\FF}{\mathbb{F}} \newcommand{\ZZ}{\mathbb{Z}} \newcommand{\NN}{\mathbb{N}} 
  \newcommand{\cV}{\mathcal{V}} \newcommand{\cW}{\mathcal{W}}   \newcommand{\med}{\;|\;} 
\DeclareMathOperator{\supp}{\mathrm{Supp}\,} \DeclareMathOperator{\Aut}{\mathrm{Aut}}     
\DeclareMathOperator{\End}{\mathrm{End}}   
\DeclareMathOperator{\bAut}{\mathbf{Aut}}  
  \DeclareMathOperator{\lspan}{\mathrm{span}} \DeclareMathOperator{\gl}{\mathfrak{gl}}
\DeclareMathOperator{\kan}{\mathfrak{K}}

  \newcommand{\GL}{\mathrm{GL}}  
 
\newcommand{\cA}{\mathcal{A}} 
  
\DeclareMathOperator{\Hom}{\mathrm{Hom}} \DeclareMathOperator{\id}{\mathrm{id}}

\DeclareMathOperator{\im}{\textnormal{im}}

 \newcommand{\cM}{\mathcal{M}}
\newcommand{\instr}{{\mathfrak{instr}}}
\newcommand{\Tr}{\mathsf{T}} 
\DeclareMathOperator{\bGL}{\mathbf{GL}} 
\DeclareMathOperator{\bG}{\mathbf{G}}

\DeclareMathOperator{\bT}{\mathbf{T}}
\DeclareMathOperator{\FLM}{\mathbf{FLM}} \DeclareMathOperator{\FLSM}{\mathbf{FLSM}}
\DeclareMathOperator{\GJP}{\mathbf{GJP}} \DeclareMathOperator{\GJSP}{\mathbf{GJSP}}

\DeclareMathOperator{\undHom}{\mathrm{\underline{Hom}}}
\DeclareMathOperator{\undEnd}{\mathrm{\underline{End}}}
\DeclareMathOperator{\undAut}{\mathrm{\underline{Aut}}}
\DeclareMathOperator{\bUndAut}{\mathbf{\underline{Aut}}}

\newtheorem{theorem}{Theorem}
\newtheorem{proposition}[theorem]{Proposition}

\theoremstyle{definition}
\newtheorem{df}[theorem]{Definition}
\newtheorem{example}[theorem]{Example}
\newtheorem{notation}[theorem]{Notation}

\theoremstyle{remark}
\newtheorem{remark}[theorem]{Remark}

\numberwithin{equation}{section} 
\numberwithin{theorem}{section} 

\begin{document}

\title[On the Faulkner construction for generalized Jordan superpairs]{On the Faulkner construction \\ for generalized Jordan superpairs}

\author[D. Aranda-Orna]{Diego Aranda-Orna}
\address{Departamento de Matem\'aticas y Computaci\'on,
Universidad de La Rioja, 26006, Logro\~no, Spain}
\email{diego.aranda.orna@gmail.com}

\thanks{The author is supported by grant MTM2017-83506-C2-1-P (AEI/FEDER, UE)}
\date{}

\begin{abstract}
In this paper, the well-known Faulkner construction is revisited and adapted to include the super case, which gives a bijective correspondence between generalized Jordan (super)pairs and faithful Lie (super)algebra (super)modules, under certain constraints (bilinear forms with properties analogous to the ones of a Killing form are required, and only finite-dimensional objects are considered). We always assume that the base field has characteristic different from $2$.

It is also proven that associated objects in this Faulkner correspondence have isomorphic automorphism group schemes. Finally, this correspondence will be used to transfer the construction of the tensor product to the class of generalized Jordan (super)pairs with ``good'' bilinear forms. 
\end{abstract}

\maketitle

\section{Introduction} \label{section.introduction}

An important connection between $\ZZ$-graded Lie algebras and other structures is given by the well-known Kantor construction. The Kantor construction applied to a Jordan (super)pair produces a $\ZZ$-graded Lie (super)algebra which is $3$-graded (that is, the support of the grading is contained in $\{-1,0,1\}$) (see \cite{K97}), and applied to Kantor pairs produces a $\ZZ$-graded Lie algebra that is $5$-graded (the support of the grading is contained in $\{-2,-1,0,1,2\}$) (see \cite{AF99}). Note that the Kantor construction for Jordan algebras, Jordan triple systems, and Jordan pairs, is usually called the Tits-Kantor-Koecher construction, or TKK-construction. A classification of the finite-dimensional simple Jordan pairs and superpairs can be found in \cite[Chap.4]{L75} and \cite{K97}, respectively. A different version of the Kantor construction, defined by a quotient of a universal graded Lie algebra (\cite{K70}, \cite{K72}, \cite[\S 3]{AK88}, \cite{AK98}), gives a bijective correspondence between the class of $\ZZ$-graded Lie algebras with a grade-reversing involution (and generated by the subspaces $L_1$ and $L_{-1}$) and the class of generalized Jordan triple systems; this has also been used in the super case \cite{BG79}, \cite[Th.~2.1]{P10} (some details of this Kantor construction are usually omitted and, unfortunately, the original references \cite{K70}, \cite{K72} are not easy to find and have not been translated from Russian). Some constructions of Lie algebras from ternary algebras appear in \cite{A76}, \cite{F71}. Kantor pairs are also called generalized Jordan pairs of second order (see \cite{KM20} and references therein). For the basic definitions related to Lie superalgebras see, e.g., \cite{CW12}, \cite{K77}.

\bigskip

The Faulkner construction has been used in Physics, where generalized Jordan triple (super)systems with ``good'' bilinear forms are referred to with different names ($M2$-branes $3$-algebras \cite{NP09}, metric Lie $3$-algebras \cite{MFMR09}, \cite{F09}, three-algebras \cite{P10}).

The original Faulkner construction \cite{F73} uses a generalized Jordan pair ``in disguise'' (one of the triple products is dropped) with a ``good'' bilinear form, and produces a Lie module where the Lie algebra has a ``good'' bilinear form, and shows how to recover the original triple product from the module (although it was not stated as a correspondence between both isomorphic classes of objects since the module was not required to be faithful). The Faulkner construction is often given in terms of triple systems instead of pairs (e.g., see \cite{F09}, \cite{MFMR09}), and has also been used to study anti-Jordan pairs \cite{FF80}. The term \textit{generalized Jordan (super)pair} has not been found explicitly in the literature.

A polished and detailed proof of the Faulkner construction, that is adapted to the super case, for pairs instead of triple systems, and stated as a correspondence, has not been found in the literature. Giving a nice reference for these results is one of the aims of this work. Note that a construction in terms of superpairs is more general than in terms of triple supersystems, because generalized Jordan superpairs can be constructed from generalized Jordan triple supersystems (and structurable superalgebras). Since the proof in \cite{F73} was rather sketchy and most computations were omitted, some related results have appeared later (e.g., see \cite{MFMR09}), and some of them are part of the proof if the construction is formulated as a bijective correspondence.

\medskip

Note that the author was studying how to extend ``good'' bilinear forms from Kantor pairs to their Kantor-Lie algebras, which led to rediscover the Faulkner construction for pairs.

\bigskip

Open problems: From an optimistic point of view, a Kantor construction for generalized Jordan (super)pairs, which includes the cases mentioned above as particular cases, is expected to exist (which perhaps is already done in the literature). If the answer is positive, then the Kantor and Faulkner constructions together would produce a Kantor-Faulkner correspondence between $\ZZ$-graded Lie superalgebras and faithful Lie supermodules, under certain constraints (``good'' bilinear forms are required, and the $\ZZ$-graded Lie superalgebra $L$ must be generated by the subspaces $L_1$ and $L_{-1}$). Giving a detailed description of this Kantor-Faulkner correspondence, with necessary and sufficient conditions, is left as an open problem. Furthermore, the Faulkner and Kantor constructions remain to be studied in the color case, that is, where Lie color algebras appear instead of Lie (super)algebras.

\bigskip

This paper is structured as follows:

\medskip

In Section~\ref{defsSection} we recall some basic definitions and introduce the notation used in further sections. A reformulation of the Faulkner construction for generalized Jordan superpairs is given in Section~\ref{sectionFaulkner}, where we also show that the automorphism group schemes are preserved in the correspondence. Finally, in Section~\ref{sectionTensor} we transfer the tensor product operator from the class of faithful Lie supermodules (with ``good'' bilinear forms) to the class of generalized Jordan superpairs (with ``good'' bilinear forms).

As an application of the original results in this paper, note that some results related to automorphism group schemes can be transferred between generalized Jordan superpairs and faithful Lie supermodules, with the restrictions given by the correspondence (e.g., classifications of gradings and orbits). (An introduction to automorphism group schemes can be found in \cite{W79}, and results that allow to transfer classifications of gradings using automorphism group schemes can be found in \cite{EKmon}.) Tensor products are a natural tool for the study of generalized Jordan superpairs (e.g., by decomposing a simple object as a tensor product of others of lesser dimension, or to construct new examples as tensor product of others).

\section{Preliminaries} \label{defsSection}

The base field $\FF$ will be always assumed to be of characteristic different from $2$. In this section we will recall the basic definitions used in further sections.

\subsection{Lie superalgebras}

Let $G$ be a group. A $G$-grading on a vector space $V$ is a vector space decomposition $\Gamma:\;V = \bigoplus_{g\in G} V_g $. If $\cA$ is an $\FF$-algebra, then a \emph{$G$-grading} on $\cA$ is a vector space decomposition $$ \Gamma:\;\cA=\bigoplus_{g\in G} \cA_g $$
such that $\cA_g \cA_h\subseteq \cA_{gh}$ for all $g,h\in G$. An algebra $\cA$ with a $G$-grading is called a {\em $G$-graded algebra}. The subspace $\cA_g$ is called {\em homogeneous component of degree $g$}. The nonzero elements $x\in\cA_g$ are said to be {\em homogeneous of degree $g$}, and we write $\deg(x) = g$. The set $\supp \Gamma := \{g\in G \med \cA_g\neq 0\}$ is called the {\em support} of the grading. 

A nonzero linear map $f \colon V \to W$ between $G$-graded vector spaces is called \emph{homogeneous of degree $g$} if $f(V_h) \subseteq W_{gh}$ for all $h\in G$; thus $\Hom(V, W)$ becomes a $G$-graded vector space. In particular, if $V$ is $G$-graded, then $\End(V)$ becomes a $G$-graded algebra.

\smallskip

A \emph{superalgebra} is a $\ZZ_2$-graded algebra $\cA = \cA_{\bar0} \oplus \cA_{\bar1}$. The homogeneous subspaces $\cA_{\bar0}$ and $\cA_{\bar1}$ are called, respectively, the \emph{even} and \emph{odd} subspaces. The degree map of the $\ZZ_2$-grading is denoted by $\varepsilon$, that is, $\varepsilon(x) := a$ if $0\neq x\in\cA_a$.
An element $0\neq x \in\cA$ is called \emph{even} (resp. \emph{odd}) if $\varepsilon(x) = \bar0$ (resp. $\varepsilon(x) = \bar 1$). We will usually write $(-1)^{\bar0} := 1$, $(-1)^{\bar1} := -1$. Also, for homogeneous elements in a $\ZZ_2$-grading we will denote
\begin{align}
\eta_{x,y} &:= (-1)^{\varepsilon(x)\varepsilon(y)}, \\
\eta_{x,y,z} &:= (-1)^{\varepsilon(x)\varepsilon(y)
+ \varepsilon(y)\varepsilon(z) + \varepsilon(z)\varepsilon(x)}.
\end{align}

A \emph{Lie superalgebra} is a superalgebra $L = L_{\bar0} \oplus L_{\bar1}$, with product denoted by $[\cdot,\cdot]$, such that 
\begin{align}
[x, y] &= - \eta_{x,y} [y,x],  \\
[x,[y,z]] &= [[x,y],z] + \eta_{x,y} [y,[x,z]],
\end{align}
for any homogeneous elements $x,y,z\in L$. Lie algebras are exactly the even Lie superalgebras (that is, with $L_{\bar1} = 0$).

Recall that if $\cA$ is an associative superalgebra, then $\cA$ becomes a Lie superalgebra with the Lie \emph{superbracket} defined by
\begin{equation}
[x,y] := xy - \eta_{x,y} yx
\end{equation}
for any homogeneous elements $x,y\in\cA$. Also, recall that if $\cA$ is a superalgebra, then $\End(\cA)$ becomes an associative superalgebra (where the homogeneous elements are the homogeneous maps). Therefore $\End(\cA)$ with the Lie superbracket becomes a Lie superalgebra, usually denoted by $\gl(m | n)$ where $\dim \cA_{\bar0} = m$ and $\dim \cA_{\bar1} = n$, or $\gl(\cA_{\bar0} | \cA_{\bar1})$.

Let $\cA$ be a superalgebra. A \emph{superderivation of degree $a$} is a homogeneous linear map $d \colon \cA \to \cA$ of degree $a\in\ZZ_2$ such that 
\begin{equation}
d(xy) = d(x)y + \eta_{d,x} x d(y)
\end{equation}
for any $x,y$ homogeneous in  $\cA$, where we write $\varepsilon(d) := a$. A \emph{superderivation} is the sum of a superderivation of degree $\bar0$ and a superderivation of degree $\bar1$. Note that if $\cA_{\bar1} = 0$, then the superderivations are exactly the derivations of $\cA$. It is well-known that the vector space of superderivations of $\cA$ becomes a Lie subsuperalgebra of $\gl(m|n)$, where $m$ and $n$ are respectively the even and odd dimensions of $\cA$.

\bigskip

Given a Lie superalgebra $L$ and a $\ZZ_2$-graded vector space $M = M_{\bar0} \oplus M_{\bar1}$ with a bilinear map $L \times M \to M$, $(x, v) \mapsto x \cdot v$, we say that $M$ is an \emph{$L$-supermodule} if $L_a \cdot M_b \subseteq M_{a+b}$ for any $a,b\in\ZZ_2$ and
\begin{equation} \label{supermodule}
[x, y] \cdot v = x \cdot (y \cdot v) - \eta_{x,y} y \cdot (x \cdot v)
\end{equation}
for any homogeneous elements $x,y\in L$, $v\in M$. Modules for Lie algebras are exactly the Lie superalgebra supermodules where $L_{\bar1} = 0$ and $M_{\bar1} = 0$.

\bigskip

Let $L$ be a Lie superalgebra and $b \colon L \times L \to \FF$ a bilinear form. We will say that $b$ is \emph{homogeneous} if $b(x,y) = 0$ for any homogeneous elements $x,y\in L$ with $\varepsilon(x) + \varepsilon(y) \neq \bar0$. If 
\begin{equation} \label{invariant}
b([x,y],z) = b(x,[y,z])
\end{equation}
for any $x,y,z\in L$, then $b$ is said to be \emph{invariant}. Also, if
\begin{equation} \label{supersymmetric_b}
b(x,y) = \eta_{x,y} b(y,x)
\end{equation}
for any homogeneous elements $x,y\in L$, then $b$ is said to be \emph{supersymmetric} (note that the restrictions of $b$ to $L_{\bar0}$ and $L_{\bar1}$ are symmetric and antisymmetric, respectively).

\bigskip

Let $L$ be a finite-dimensional Lie algebra and $M$ a finite-dimensional $L$-module. Recall that the \emph{dual $L$-module} is given by the dual space $M^*$ with the dual action $x \cdot f$ defined by
\begin{equation}
(x \cdot f)(v) := - f(x\cdot v)
\end{equation}
for any $x\in L$, $f\in M^*$, $v\in M$.

\smallskip

Now assume that $L$ is a finite-dimensional Lie superalgebra and $M$ a finite-dimensional $L$-supermodule. Consider the dual space $M^*$. Then $M^*$ inherits a dual $\ZZ_2$-grading such that the duality bilinear form is homogeneous (that is, $M_a$ and $M^*_a$ are paired and $M_a$ is orthogonal to $M^*_b$ if $a \neq b$). We will usually denote by $\langle\cdot,\cdot\rangle \colon M^* \times M \to \FF$ the pairing bilinear form. Define the \emph{left-dual} (or \emph{dual}) \emph{$L$-supermodule} of $M$ as the $\ZZ_2$-graded vector space $M^*$ with the dual action $x \cdot f$ given by
\begin{equation} \label{dualityAction}
\langle x \cdot f, v \rangle =
(x \cdot f)(v) := -\eta_{x,f} f(x\cdot v)
= -\eta_{x,f} \langle f, x\cdot v \rangle
\end{equation}
for any homogeneous elements $x\in L$, $f\in M^*$, $v\in M$. It will be denoted by $M^\gets$ or $M^*$. Similarly, define the \emph{right-dual $L$-supermodule} $M^\to$ of $M$ as the $\ZZ_2$-graded vector space $M^*$ with the dual action $x \cdot f$ given by
\begin{equation}
\langle x \cdot f, v \rangle =
(x \cdot f)(v) := -\eta_{x,v} f(x\cdot v)
= -\eta_{x,v} \langle f, x\cdot v \rangle
\end{equation}
for any homogeneous elements $x\in L$, $f\in M^*$, $v\in M$. It is easy to see that $(M^\gets)^\to \cong M \cong (M^\to)^\gets$ (left and right duals are inverses) and $(((M^\gets)^\gets)^\gets)^\gets \cong M$ and $(((M^\to)^\to)^\to)^\to \cong M$ (both dualizations have order dividing $4$). In particular, when $L_{\bar1} = 0$ and $M_{\bar1} = 0$, we have that $M^\gets \cong M^\to$.

Given a homogeneous element $\varphi\in\End(M)$, we define its \emph{left-dual} map $\varphi^\gets$ by means of
\begin{equation}
\langle \varphi^\gets(f), v \rangle = \eta_{\varphi, f} \langle f, \varphi(v) \rangle,
\end{equation}
for any homogeneous $f\in M^*$, $v\in M$. The map $\varphi^\gets$ may also be referred to as the \emph{dual} map of $\varphi$, and denoted by $\varphi^*$. The dual map $\varphi^\gets$ is well-defined for any $\varphi\in\End(M)$, since any endomorphism decomposes as sum of homogeneous elements. Similarly, we define the \emph{right-dual} of $\varphi$ by means of 
\begin{equation}
\langle \varphi^\to(f), v \rangle = \eta_{\varphi, v} \langle f, \varphi(v) \rangle,
\end{equation}
for any homogeneous $f\in M^*$, $v\in M$, $\varphi\in\End(M)$. In the case that $M_{\bar1} = 0$, we have that both $\varphi^\gets$ and $\varphi^\to$ coincide with the usual dual map $\varphi^*$. Also, it is clear that $(\varphi^\gets)^\to = \varphi = (\varphi^\to)^\gets$ (left and right dualizations are inverse processes). Besides $(((\varphi^\gets)^\gets)^\gets)^\gets = \varphi$ and $(((\varphi^\to)^\to)^\to)^\to = \varphi$ (both dualizing processes have order dividing $4$). Note that dualizing preserves parity of homogeneous endomorphisms.

\bigskip

If $M$ and $N$ are $L$-supermodules, the \emph{(left) tensor product supermodule} is defined by the vector space $M \otimes N$ with the action
\begin{equation} \label{leftTensorMods}
x \cdot (v \otimes w) := (x \cdot v) \otimes w + \eta_{x,v} v \otimes (x \cdot w)
\end{equation}
for each homogeneous $x\in L$, $v\in M$, $w \in N$. It will be denoted by $M \overleftarrow{\otimes} N$ or $M \otimes N$. The parity map that determines the $\ZZ_2$-grading is given by $\varepsilon(v \otimes w) := \varepsilon(v) + \varepsilon(w)$. The tensor product operator is associative, but $M \otimes N \ncong N \otimes M$ (unless $M_{\bar1} = 0 = N_{\bar1}$ or $L_{\bar1} = 0$). The \emph{right tensor product} can be defined similarly, and denoted by $M \overrightarrow{\otimes} N$. Note that $M \overleftarrow{\otimes} N \cong N \overrightarrow{\otimes} M$. Besides, $M \overleftarrow{\otimes} N = M \overrightarrow{\otimes} N$ if either $M_{\bar1} = 0 = N_{\bar1}$ or $L_{\bar1} = 0$.

\smallskip

It is easy to see that $(M \otimes N)^\gets \cong M^\gets \otimes N^\gets$. To show this, consider the bilinear form given by
\begin{equation} \label{tensorProductBilinearForms} \begin{split}
\langle\cdot,\cdot\rangle \colon & (M^\gets \otimes N^\gets) \times (M \otimes N) \to \FF \\
& \langle f \otimes g, v \otimes w \rangle := \eta_{g,v} \langle f,v \rangle \langle g,w \rangle,
\end{split} \end{equation}
where $f \in M^\gets$, $g\in N^\gets$, $v\in M$, $w\in N$ are homogeneous.
The bilinear map in \eqref{tensorProductBilinearForms} will be called the \emph{(left) tensor superproduct} of the corresponding pairing bilinear forms $M^* \times M \to \FF$ and $N^* \times N \to \FF$.
Then we have that 
\begin{align*}
\langle f \otimes g & , x \cdot (v \otimes w) \rangle =
\langle f \otimes g, (x \cdot v) \otimes w + \eta_{x,v} v \otimes (x \cdot w) \rangle \\
&= \eta_{g, x\cdot v} \langle f, x \cdot v\rangle \langle g, w \rangle
+ \eta_{x,v}\eta_{g,v} \langle f,v \rangle \langle g, x\cdot w \rangle \\
&= - \eta_{x,g}\eta_{g,v}\eta_{x,f} \langle x \cdot f, v \rangle \langle g,w \rangle
- \eta_{x,v}\eta_{g,v}\eta_{x,g} \langle f,v \rangle \langle x \cdot g, w \rangle \\
&= - (\eta_{x,g}\eta_{g,v}\eta_{x,f})\eta_{g,v}
\langle (x \cdot f) \otimes g, v \otimes w \rangle \\
& \quad - (\eta_{x,v}\eta_{g,v}\eta_{x,g})(\eta_{x,v}\eta_{g,v})
\langle f \otimes (x\cdot g), v \otimes w \rangle \\
&= - \langle \eta_{x,g}\eta_{x,f} (x \cdot f) \otimes g
+ \eta_{x,g} f \otimes (x \cdot g), v\otimes w \rangle \\
&= - \eta_{x,f}\eta_{x,g} \langle x \cdot (f \otimes g), v \otimes w \rangle,
\end{align*}
for any homogeneous elements $x \in L$, $f \in M^\gets$, $g\in N^\gets$, $v\in M$, $w\in N$,
which proves the isomorphism.

\bigskip

Let $M$ an $L$-supermodule and $M'$ an $L'$-supermodule, with $L$ and $L'$ Lie superalgebras. For $a\in\ZZ_2$, define
$\undHom^a\big((L,M),(L',M')\big)$ as the set consisting of pairs $(\varphi_0, \varphi^+)$, where $\varphi_0 \colon L \to L'$ is a superalgebra homomorphism (that is, an even algebra homomorphism), $\varphi^+ \colon M \to M'$ is a linear map of parity $a$ (that is, $\varphi^+(M_b) \subseteq M'_{a+b}$ for each $b\in\ZZ_2$), and where we also have
\begin{equation}
\varphi^+(x \cdot v) = \eta_{\varphi^+, x} \varphi_0(x) \cdot \varphi^+(v).
\end{equation}
Also, denote
$$\undHom\big((L,M),(L',M')\big) := \bigcup_{a\in\ZZ_2} \undHom^a\big((L,M),(L',M')\big),$$
and $\Hom\big((L,M),(L',M')\big) := \undHom^{\bar0}\big((L,M),(L',M')\big)$. In the case that $L = L'$, the subset $\undHom^a_L(M, M') \subseteq \undHom^a\big((L,M),(L,M')\big)$ given by the elements where $\varphi_0 = \id_L$ is clearly a vector space. Thus
$$ \undHom_L(M, M') := \undHom^{\bar0}_L(M, M') \oplus \undHom^{\bar1}_L(M, M') $$
is a $\ZZ_2$-graded vector space. The elements of the vector space $\Hom_L(M, M') := \undHom^{\bar 0}_L(M, M')$ will be called \emph{$L$-supermodule homomorphisms}.

\smallskip

It is well-known that if $M$, $N$ are $L$-supermodules for a Lie superalgebra $L$, then the vector space of linear maps $M \to N$ inherits a $\ZZ_2$-grading,
$\Hom(M, N) = \Hom^{\bar 0}(M, N) \oplus \Hom^{\bar 1}(M, N)$,
and it becomes an $L$-supermodule with the action
\begin{equation}
(x \cdot f)(v) := x \cdot f(v) - \eta_{x,f} f(x \cdot v)
\end{equation}
for $x \in L$, $f \in \Hom(M,N)$, $v \in M$.

\bigskip

There are several definitions of automorphism groups for Lie supermodules. It is clear that the set of bijective elements in $\undEnd(L,M) := \undHom\big((L,M),(L,M)\big)$ defines a group, that will be denoted by $\undAut(L,M)$. Note that $\undAut(L,M) \leq \Aut(L)\times\GL(M)$. The even elements define a subgroup
$\Aut(L,M) := \undAut^{\bar0}(L,M) \unlhd \undAut(L,M)$. On the other hand, the elements with $\varphi_0 = \id_L$ define a subgroup $\undAut_L(M) \unlhd \undAut(L,M)$. Set $\Aut_L(M) := \undAut_L(M) \cap \Aut(L,M)$. The automorphism group schemes $\bUndAut(L,M)$, $\bUndAut_L(M)$, $\bAut(L,M)$ and $\bAut_L(M)$ are defined similarly.

\subsection{Generalized Jordan superpairs}

Recall from \cite[\S3]{AFS17} that a \emph{trilinear pair} is a pair of vector spaces $\cV = (\cV^-, \cV^+)$ with a pair of trilinear maps $\{\cdot,\cdot,\cdot\}^\sigma \colon \cV^\sigma \times \cV^{-\sigma} \times \cV^\sigma \to \cV^\sigma$, $\sigma \in \{+, -\}$. We will write
\begin{equation} 
D^\sigma_{x,y}(z) := \{x,y,z\}^\sigma
\end{equation}
for $x,z\in \cV^\sigma$, $y\in \cV^{-\sigma}$, $\sigma = \pm$. The superscript $\sigma$ is sometimes omitted for short.

Let $G$ be an abelian group and $\cV$ a trilinear pair. Given two decompositions of vector spaces $\Gamma^{\sigma} \colon \cV^\sigma = \bigoplus_{g\in G} \cV_g^{\sigma}$, for $\sigma = \pm$, we will say that $\Gamma=(\Gamma^+,\Gamma^-)$ is a {\em $G$-grading on $\cV$} if $\lbrace \cV^{\sigma}_g, \cV^{-\sigma}_h, \cV^{\sigma}_k \rbrace \subseteq \cV^{\sigma}_{g+h+k}$ for any $g,h,k\in G$ and $\sigma\in \lbrace +,- \rbrace$. The vector space $\cV^+_g \oplus \cV^-_g$ is called the {\em homogeneous component of degree} $g$. If $0\neq x\in \cV^{\sigma}_g$ we say $x$ is {\em homogeneous of degree $g$}.
Notice that if $x,y$ are homogeneous, then $D^\sigma_{x,y}$ is a homogeneous map of degree $\varepsilon(D^\sigma_{x,y}) := \varepsilon(x) + \varepsilon(y)$.

\bigskip

A \emph{generalized Jordan superpair} is a trilinear pair $\cV = (\cV^-, \cV^+)$, where the subspaces $\cV^-$ and $\cV^+$ are $\ZZ_2$-graded and we have that
\begin{equation} \label{defGJSP}
\begin{split}
D^\sigma_{x,y} & D^\sigma_{z,w} - 
(-1)^{(\varepsilon(x) + \varepsilon(y))(\varepsilon(z) + \varepsilon(w))}
D^\sigma_{z,w} D^\sigma_{x,y} \\
&= D^\sigma_{D^\sigma_{x,y}z, w} - 
(-1)^{\varepsilon(x)\varepsilon(y) + \varepsilon(y)\varepsilon(z) + \varepsilon(z)\varepsilon(x)}
D^\sigma_{z, D^{-\sigma}_{y,x}w}
\end{split}
\end{equation}
for any homogeneous elements $x,z\in \cV^\sigma$, $y,w\in \cV^{-\sigma}$, $\sigma = \pm$. Note that the left side of \eqref{defGJSP} is just the Lie superbracket
\begin{equation*}
[D^\sigma_{x,y}, D^\sigma_{z,w}] = D^\sigma_{x,y} D^\sigma_{z,w} - 
\eta_{D^\sigma_{x,y}, D^\sigma_{z,w}} D^\sigma_{z,w} D^\sigma_{x,y}.
\end{equation*}
Thus \eqref{defGJSP} is equivalent to
\begin{equation} \label{eqGJSP}
[D^\sigma_{x,y}, D^\sigma_{z,w}] = 
D^\sigma_{D^\sigma_{x,y}z, w} - \eta_{x,y,z} D^\sigma_{z, D^{-\sigma}_{y,x}w}.
\end{equation}
In particular, if $\cV^\sigma_{\bar1} = 0$ for $\sigma = \pm$, then $\cV$ is called a \emph{generalized Jordan pair}. On the other hand, if $\cV^\sigma_{\bar0} = 0$ for $\sigma = \pm$, then $\cV$ is called a \emph{generalized Jordan antipair}. Sometimes \eqref{eqGJSP} is referred to as the \emph{fundamental identity}.

\bigskip

Let $\cV$ be a generalized Jordan superpair, $D = (D^-, D^+) \in \End(\cV^-) \times \End(\cV^+)$, and $a\in\ZZ_2$. We will say that $D$ is a \emph{superderivation of degree $a$} of $\cV$, and write $\varepsilon(D) := a$, if we have that $D^\sigma \cV^\sigma_b \subseteq \cV^\sigma_{a+b}$ for any $\sigma=\pm$, $b\in\ZZ_2$, and
\begin{equation} \label{derivGJSPdef1}
D^\sigma(\{x,y,z\}) = \{D^\sigma(x),y,z\} + \eta_{D, x} \{x,D^{-\sigma}(y),z\}
+ \eta_{D, D_{x,y}} \{x,y,D^\sigma(z)\}
\end{equation}
for any homogeneous elements $x,z\in\cV^\sigma$, $y\in\cV^{-\sigma}$. A \emph{superderivation} is the sum of an even superderivation and an odd superderivation. 

\smallskip

Given a generalized Jordan superpair $\cV$, consider the operators
\begin{equation}
\nu(x,y) := (D^-_{x,y}, - \eta_{x,y} D^+_{y,x}) \in\End(\cV^-) \times \End(\cV^+),
\end{equation}
for homogeneous elements $x\in\cV^-$, $y\in\cV^+$. For convenience, we will also denote 
\begin{equation} \label{nuAntisymmetry}
\nu(y, x) := - \eta_{x,y} \nu(x, y).
\end{equation}
By \eqref{defGJSP} or \eqref{eqGJSP}, we have
\begin{align*}
D_{x,y} & \{u,v,w\} = \{D_{x,y}u,v,w\} - \eta_{x,y,u} \{u,D_{y,x}v,w\}
+ \eta_{D_{x,y}, D_{u,v}} \{u,v,D_{x,y}w\} \\
&= \{D_{x,y}u,v,w\} - \eta_{x,y}\eta_{D_{x,y},u} \{u,D_{y,x}v,w\}
+ \eta_{D_{x,y}, D_{u,v}} \{u,v,D_{x,y}w\},
\end{align*}
so that the operators $\nu(x,y)$ are superderivations. Also note that
\begin{equation} \label{nuBracket1} \begin{split}
[\nu(x,y)&, \nu(z,w)] = \nu(D_{x,y}z,w) - \eta_{x,y,z} \nu(z,D_{y,x}w) \\
&= \nu(\nu(x,y) \cdot z,w) + \eta_{z, D_{x,y}} \nu(z, \nu(x,y) \cdot w) .
\end{split} \end{equation}
Denote $\cV_{a} := \cV^-_a \oplus \cV^+_a$ for $a\in\ZZ_2$.
The \emph{inner structure (Lie) superalgebra} of a generalized Jordan superpair is the Lie superalgebra
\begin{equation}
\instr(\cV) := \lspan\{ \nu(x,y) \med x\in\cV^-, y\in\cV^+ \} \leq \gl(\cV_{\bar0} | \cV_{\bar1}),
\end{equation}
and its elements are called \emph{inner superderivations} of $\cV$. (Inner structure algebras appear in \cite{F73}, and are also used in the Kantor construction.) If $\cV^\sigma_{\bar1} = 0$ for $\sigma = \pm$, then $\instr(\cV)$ is a Lie algebra called the \emph{inner structure algebra} of $\cV$, and its elements are called \emph{inner derivations}. From \eqref{nuBracket1}, it is clear that
\begin{equation} \label{nuBracket2}
[x, \nu(f,v)] = \nu(x \cdot f,v) + \eta_{x,f} \nu(f, x \cdot v) .
\end{equation}
for any homogeneous elements $x\in \instr(\cV)$, $f\in \cV^-$, $v\in\cV^+$.

\bigskip

A generalized Jordan superpair $\cV$ is called a \textit{Jordan superpair} if
\begin{equation}
\{x,y,z\}^\sigma = \eta_{x,y} \eta_{x,z} \eta_{y,z} \{z,y,x\}^\sigma
\end{equation}
for all $x,z\in\cV^\sigma$, $y\in\cV^{-\sigma}$, $\sigma = \pm$. If a Jordan superpair is even ($\cV_{\bar1}^\sigma = 0$ for $\sigma = \pm$), it is called a \textit{Jordan pair}. If a Jordan superpair is odd ($\cV_{\bar0}^\sigma = 0$ for $\sigma = \pm$), it is called a \textit{Jordan antipair} or \textit{anti-Jordan pair} (this is equivalent to the definition in \cite{BEM11}).

\smallskip

Given a generalized Jordan pair $\cV$, we say that $\cV$ is a \emph{Kantor pair} (or \emph{generalized Jordan pair of second order}) if
\begin{equation}
K^\sigma_{K^\sigma_{x,y}z,w} = K^\sigma_{x,y}D^{-\sigma}_{z,w} + D^\sigma_{w,z}K^\sigma_{x,y}
\end{equation}
for all $x,z \in\cV^\sigma$, $y,w\in\cV^{-\sigma}$, $\sigma = \pm$.
It is well-known that Jordan pairs are exactly the Kantor pairs satisfying $K_{x,y} = 0$ for all $x,y\in\cV^\sigma$, $\sigma = \pm$.

\bigskip

Let $\cV = (\cV^+, \cV^-)$ be a generalized Jordan superpair with a bilinear form $\langle\cdot,\cdot\rangle \colon \cV^- \times \cV^+ \to \FF$. We will say that $\langle\cdot,\cdot\rangle$ is \emph{left-superinvariant}, or \emph{superinvariant}, if
\begin{equation} \label{superinvariant}
\langle D_{x,y}z, w \rangle = \eta_{x,y,z} \langle z, D_{y,x}w \rangle
\end{equation}
for any homogeneous elements $x,z\in\cV^-$, $y,w\in\cV^+$. Similarly, $\langle\cdot,\cdot\rangle$ is said to be \emph{right-superinvariant} if
\begin{equation}
\langle D_{x,y}z, w \rangle = \eta_{x,y,w} \langle z, D_{y,x}w\rangle .
\end{equation}
In particular, if $\cV_{\bar1} = 0$, then the bilinear form $\langle\cdot,\cdot\rangle$ is left-invariant if and only if it is right-invariant, and in this case it will be said to be \emph{invariant}.
If $\langle x, y \rangle = 0$ for any $x\in\cV^-$, $y\in\cV^+$ such that $\varepsilon(x) \neq \varepsilon(y)$, then the bilinear form is said to be \emph{homogeneous}. On the other hand, we will say that $\langle\cdot,\cdot\rangle$ is \emph{left-supersymmetric}, or \emph{supersymmetric}, if it satisfies the identities
\begin{equation}\label{symmetricForm}
\begin{split}
\langle D_{x,y}z, w \rangle &= \eta_{D_{x,y}, D_{z,w}} \langle D_{z,w}x, y \rangle, \\
\langle x, D_{y,z}w \rangle &= \eta_{D_{x,y}, D_{z,w}} \langle z, D_{w,x}y \rangle
\end{split}
\end{equation}
for any homogeneous elements $x,z\in\cV^-$, $y,w\in\cV^+$.
Note that the two identities in \eqref{symmetricForm} are equivalent if $\langle\cdot,\cdot\rangle$ is either left-superinvariant or right-superinvariant. Similarly, we  will say that $\langle\cdot,\cdot\rangle$ is \emph{right-supersymmetric} if it satisfies the identities
\begin{equation}
\begin{split}
\langle D_{x,y}z, w \rangle &= \eta_{D_{x,w}, D_{z,y}} \langle D_{z,w}x, y \rangle, \\
\langle x, D_{y,z}w \rangle &= \eta_{D_{x,w}, D_{z,y}} \langle z, D_{w,x}y \rangle
\end{split}
\end{equation}
for any homogeneous elements $x,z\in\cV^-$, $y,w\in\cV^+$. Note that if $\cV_{\bar1} = 0$, then the left and right supersymmetric identities coincide, and we will say in this case that the bilinear form is \emph{symmetric}.

\bigskip

Recall that a homomorphism $\varphi \colon \cV \to \cW$ of generalized Jordan pairs is a pair of linear maps $\varphi = (\varphi^-, \varphi^+)$, with $\varphi^\sigma \colon \cV^\sigma \to \cW^\sigma$, such that $\varphi^\sigma(\{x,y,z\}^\sigma) = \{ \varphi^\sigma(x), \varphi^{-\sigma}(y), \varphi^\sigma(z) \}$ for any $x,z\in\cV^\sigma$, $y\in\cV^{-\sigma}$, $\sigma = \pm$. For generalized Jordan superpairs, we also require that the homomorphisms preserve the parity. As usual, the automorphism group of $\cV$ will be denoted as $\Aut(\cV)$, and the automorphism group scheme as $\bAut(\cV)$.

\section{The Faulkner construction for GJSP} \label{sectionFaulkner}

In this section we will revisit the Faulkner correspondence, giving a detailed proof that is adapted to generalized Jordan superpairs.

\begin{notation}
We will denote by $\FLSM$ the class of objects of the form $(L, M, b)$, where $L$ is a finite-dimensional Lie superalgebra, $M$ is a finite-dimensional faithful $L$-supermodule, and $b \colon L \times L \to \FF$ is a nondegenerate homogeneous invariant supersymmetric bilinear form. We will denote by $\Aut(L, b)$ the subgroup of $\Aut(L)$ preserving the bilinear form, which consists of the elements $\varphi_0 \in \Aut(L)$ such that
\begin{equation}
b(\varphi_0(x), \varphi_0(y)) = b(x,y)
\end{equation}
for any $x,y\in L$. Also, we will denote by $\Aut(L, M, b)$ the group of automorphisms of the supermodule preserving the bilinear form, which consists of the pairs $\varphi = (\varphi_0, \varphi^+) \in \Aut(L, M)$ such that $\varphi_0 \in \Aut(L, b)$. The subgroup schemes $\bAut(L, b) \leq \bAut(L)$ and $\bAut(L, M, b) \leq \bAut(L, M)$ are defined similarly. The subclass of objects of $\FLSM$ where $L_{\bar1} = 0$ and $M_{\bar1} = 0$ will be denoted by $\FLM$.

Also, denote by $\GJSP$ the class of objects of the form $(\cV, \langle\cdot,\cdot\rangle)$, where $\cV$ is a finite-dimensional generalized Jordan superpair with a nondegenerate homogeneous superinvariant supersymmetric bilinear form $\langle\cdot,\cdot\rangle \colon \cV^- \times \cV^+ \to \FF$. We will denote by $\Aut(\cV, \langle\cdot,\cdot\rangle)$ the group of automorphisms preserving the bilinear form, which consists of the the pairs $\varphi = (\varphi^-, \varphi^+) \in \Aut(\cV)$ such that
\begin{equation}
\langle \varphi^-(f), \varphi^+(v) \rangle = \langle f, v \rangle
\end{equation}
for any $f\in \cV^-$, $v\in \cV^+$. Again, the subgroup scheme $\bAut(\cV, \langle\cdot,\cdot\rangle) \leq \bAut(\cV)$ is defined similarly. The subclass of objects of $\GJSP$ where $\cV_{\bar1} = 0$ will be denoted by $\GJP$.
\end{notation}


Next result generalizes part of \cite[Lemma~1.1]{F73} to the super case:

\begin{proposition} \label{GJP}
Let $L$ be a finite-dimensional Lie superalgebra, $M$ a finite-dimensional $L$-supermodule, and $b$ a nondegenerate homogeneous invariant supersymmetric bilinear form on $L$. Let $M^* := M^\gets$ denote the (left) dual $L$-supermodule of $M$. Consider the pairing bilinear form $\langle \cdot, \cdot \rangle \colon M^* \times M \to \FF$ given by $\langle f, v\rangle := f(v)$ for $v\in M$, $f\in M^*$, and also denote
\begin{equation} \label{flipDualBilinearForm}
\langle v, f\rangle := \eta_{f,v} \langle f, v\rangle = \eta_{f,v} f(v).
\end{equation}
For each $v\in M$, $f\in M^*$, denote by $[f, v], [v, f] \in L$ the only elements satisfying
\begin{equation} \label{bilinearFormsEc}
b(x, [f,v]) = \langle x \cdot f, v \rangle \quad \text{and}
\quad b(x, [v, f]) = \langle x \cdot v, f \rangle
\end{equation}
for all $x\in L$. Then $\cV_{L,M} := (M^*, M)$ becomes a generalized Jordan superpair with triple products given by
\begin{equation} \label{tripleProductsFromAction}
\{f, v, g\}^- := [f,v] \cdot g, \qquad \{v, f, w\}^+ := [v,f] \cdot w,
\end{equation}
for $f,g\in M^*$, $v,w\in M$. Besides, $\langle\cdot,\cdot\rangle$ is a nondegenerate homogeneous superinvariant supersymmetric bilinear form on $\cV_{L,M}$.
\end{proposition}
\begin{proof}
First, we claim that $[f,v]$ and $[v,f]$ are well-defined. For each $f\in M^*$, $v\in M$, consider the map $\psi_{f,v}\colon L \to \FF$, $x \mapsto \langle x \cdot f, v \rangle$. Since $\psi_{f,v} \in L^*$ and $b$ is nondegenerate, there exists a unique element $y\in L$ such that $\psi_{f,v}(x) = b(x, y)$ for all $x\in L$, which proves the claim for the former case, and the latter case is analogous. It is also clear that the expressions $[f,v]$ and $[v,f]$ are bilinear on their parameters. Since $b$ and $\langle\cdot,\cdot\rangle$ are homogeneous and $b$ is nondegenerate, it follows that the brackets defined in \eqref{bilinearFormsEc} are homogeneous too, that is,
\begin{equation} \label{bracketsHomogeneous}
\varepsilon([f,v]) = \varepsilon(f) + \varepsilon(v) = \varepsilon([v,f]).
\end{equation}

We need to prove that the identities
\begin{align}
[D^-_{f,v}, D^-_{g,w}] &= D^-_{D^-_{f,v}g, w} - \eta_{f,v,g} D^-_{g, D^+_{v,f}w}, \label{particularKantorIdentities} \\
[D^+_{v,f}, D^+_{w,g}] &= D^+_{D^+_{v,f}w, g} - \eta_{v,f,w} D^+_{w, D^-_{f,v}g}, \label{particularKantorIdentities2}
\end{align}
hold for any $v,w\in M$, $f,g\in M^*$.

From \eqref{bilinearFormsEc} and \eqref{dualityAction}, it follows that
$$ b(x, [f, v]) = \langle x \cdot f, v \rangle = - \eta_{x,f} \langle f, x \cdot v \rangle
= - \eta_{x,f} \eta_{x\cdot v,f} \langle x \cdot v, f \rangle = b(x, - \eta_{f,v} [v,f])$$
for all homogeneous $x\in L$, $f\in M^*$, $v \in M$, and since $b$ is nondegenerate, we get that
\begin{equation} \label{antisymmetry_fv}
[f, v] = - \eta_{f,v} [v, f]
\end{equation}
for all homogeneous $f\in M^*$, $v \in M$.

For each homogeneous elements $x\in L$, $f,g\in M^*$, $v,w\in M$, we have that
\begin{align*}
b(x&, \big[[f,v],[g,w]\big]) =_{\eqref{invariant}} \\
&= b(\big[x,[f,v]\big], [g,w]) =_{\eqref{bilinearFormsEc}} \\
&= \langle \big[x,[f,v]\big] \cdot g, w \rangle =_{\eqref{supermodule}} \\
&= \langle x \cdot ([f,v] \cdot g), w \rangle
- \eta_{x, [f,v]} \langle [f,v] \cdot (x \cdot g), w \rangle
=_{\eqref{bilinearFormsEc}, \eqref{dualityAction}} \\
&= b(x, \big[[f,v] \cdot g, w\big])
+ \eta_{x,f}\eta_{x,v} \eta_{[f,v], x\cdot g} \langle x \cdot g, [f,v] \cdot w \rangle
=_{\eqref{bilinearFormsEc}, \eqref{antisymmetry_fv}} \\
&= b(x, \big[[f,v] \cdot g, w\big] - (\eta_{f,g}\eta_{v,g})\eta_{f,v} \big[g, [v, f] \cdot w\big]) \\
&= b(x, [\{f,v,g\}, w] - \eta_{f,v,g} [g, \{v,f,w\}]),
\end{align*}
and since $b$ is nondegenerate it follows that
\begin{equation} \label{bracketOperatorsActing}
\big[[f,v],[g,w]\big] = [D_{f,v}g, w] - \eta_{f,v,g} [g, D_{v,f}w].
\end{equation}
Similarly, and using \eqref{flipDualBilinearForm}, we get that
\begin{equation} \label{bracketOperatorsActing2}
\big[[v,f],[w,g]\big] = [D_{v,f}w, g] - \eta_{v,f,w} [w, D_{f,v}g],
\end{equation}
although this also follows from
\begin{align*}
\big[[ & v,f],[w,g]\big] =_\text{\eqref{antisymmetry_fv}} \\
&= \eta_{f,v}\eta_{g,w} \big[[f,v],[g,w]\big] =_\text{\eqref{bracketOperatorsActing}} \\
&= \eta_{f,v}\eta_{g,w} \big( [D_{f,v}g, w] - \eta_{f,v,g} [g, D_{v,f}w] \big) =_\text{\eqref{antisymmetry_fv}} \\
&= (\eta_{f,v}\eta_{g,w}) (- \eta_{f,w}\eta_{v,w}\eta_{g,w}) [w, D_{f,v}g] \\
& \quad - (\eta_{f,v}\eta_{g,w}) \eta_{f,v,g} (- \eta_{g,v}\eta_{g,f}\eta_{g,w}) [D_{v,f}w, g] \\
&= [D_{v,f}w, g] - \eta_{v,f,w} [w, D_{f,v}g].
\end{align*}
Note that the actions on $M^*$ given by the elements of each side of Eq.~\eqref{bracketOperatorsActing} coincide with the actions given by the operators on each side of Eq.~\eqref{particularKantorIdentities}, from where it follows that Eq.~\eqref{particularKantorIdentities} does hold. Similarly, Eq.~\eqref{particularKantorIdentities2} does hold too. Thus we have proven that $\cV_{L,M}$ is indeed a generalized Jordan superpair.

On the other hand, from \eqref{dualityAction} and \eqref{antisymmetry_fv}, it follows that for homogeneous $f,g\in M^*$, $v,w\in M$ we have
\begin{align*}
\langle D_{f,v}g, w \rangle &= \langle [f,v] \cdot g, w \rangle
= - \eta_{[f,v],g} \langle g, [f,v] \cdot w \rangle \\
&= \eta_{f,g}\eta_{v,g}\eta_{f,v} \langle g, [v,f] \cdot w \rangle
= \eta_{f,v,g} \langle g, D_{v,f}w \rangle, 
\end{align*}
that is, the bilinear form $\langle\cdot,\cdot\rangle$ is superinvariant.
Since $b$ is supersymmetric, we get
\begin{align*}
\langle D_{f,v}g, w \rangle &= \langle [f,v] \cdot g, w \rangle
= b([f,v],[g,w]) = \eta_{[f,v], [g,w]} b([g,w],[f,v]) \\
&= \eta_{D_{f,v}, D_{g,w}} \langle [g,w] \cdot f, v \rangle
= \eta_{D_{f,v}, D_{g,w}} \langle D_{g,w}f, v \rangle,
\end{align*}
and therefore $\langle\cdot,\cdot\rangle$ is also supersymmetric. 
\end{proof}


The following result is a generalization of \cite[Lemma~23]{MFMR09} to the super case:

\begin{proposition} \label{instrTheorem}
Under the assumptions of Prop.~\ref{GJP},
$$ \instr(L, M) := \lspan\{ [f,v] \med f\in M^*, v\in M \} $$
is an ideal of $L$, and the restriction of the representation $\Phi \colon L \to \gl(M^* \oplus M)$ defines an epimorphism of superalgebras given by
\begin{equation}\begin{split}
\Upsilon \colon \instr(L, M) &\longrightarrow \instr(\cV_{L,M}) \leq \gl(M^* \oplus M), \\
\quad [f,v] &\longmapsto \nu(f,v) := (D_{f,v}, - \eta_{f,v} D_{v,f}).
\end{split}\end{equation}
Furthermore, $\ker\Phi = \instr(L, M)^\perp$.
In particular, if the $L$-supermodule $M$ is faithful, then $L = \instr(L, M) \cong \instr(\cV_{L,M})$.
\end{proposition}
\begin{proof}
It follows from \eqref{bracketOperatorsActing} and \eqref{bracketsHomogeneous} that $\instr(L, M)$ is a Lie subsuperalgebra of $L$. Moreover, for each homogeneous $x,y\in L$, $f\in M^*$, $v\in M$, we have that
\begin{align*}
b(y, \big[x,[f,v]\big]) &= b([y,x], [f,v]) = \langle [y,x] \cdot f, v \rangle \\
&= \langle y \cdot (x \cdot f), v \rangle - \eta_{x,y} \langle x \cdot (y \cdot f), v \rangle \\
&= b(y, [x \cdot f, v]) + \eta_{x,y}\eta_{x,y\cdot f} \langle y \cdot f, x \cdot v \rangle \\
&= b(y, [x \cdot f, v] + \eta_{x,f} [f, x \cdot v]),
\end{align*}
with $b$ nondegenerate, which implies that
\begin{equation}
\big[x, [f,v]\big] = [x\cdot f, v] + \eta_{x,f} [f, x\cdot v],
\end{equation}
and therefore $\instr(L, M)$ is an ideal of $L$. It is clear that the action of $[f,v]$ on $M^* \oplus M$ is given by $\nu(f,v)$, and since $\instr(\cV_{L,M})$ is spanned by the operators $\nu(f,v)$, it follows that $\Upsilon$ is an epimorphism. (Also, note that the notation in \eqref{nuAntisymmetry} is consistent with \eqref{antisymmetry_fv}.)

Set $K = \instr(L,M)^\perp := \{ k \in L \med b(k, \instr(L,M)) = 0 \}$. Note that since $b$ is homogeneous and $\instr(L,M)$ is a subsuperalgebra of $L$, it follows that we can decompose $K = K_{\bar0} \oplus K_{\bar1}$ with $K_a \subseteq L_a$.
For each homogeneous $k \in K$, $f\in M^*$, $v\in M$,  we have
$$ \langle f, k\cdot v \rangle = - \eta_{k,f} \langle k\cdot f, v \rangle
= - \eta_{k,f} b(k, [f,v]) = 0 $$
with $\langle\cdot,\cdot\rangle$ nondegenerate, hence $k\cdot v = 0$ and $k\cdot f = 0$, which shows that $K \subseteq \ker\Phi$. On the other hand, for each $z\in\ker\Phi$, $f\in M^*$, $v\in M$, we have that $b(z, [f,v]) = \langle z\cdot f, v\rangle = 0$, so $\ker\Phi \subseteq \instr(L,M)^\perp = K$. We have proven that $K = \ker\Phi$.

Finally, assume that the representation $\Phi$ is faithful. Then $\Upsilon$ is an isomorphism. Since $K = \instr(L,M)^\perp = \ker\Phi = 0$, the restriction of $b$ to $\instr(L,M)$ must be nondegenerate and $L = \instr(L,M)$.
\end{proof}

\begin{df}
Under the assumptions in Propositions~\ref{GJP} and \ref{instrTheorem}, we will say that $\cV_{L,M}$ is the \emph{generalized Jordan superpair associated to the $L$-supermodule $M$}, and the ideal $\instr(L, M) \unlhd L$ will be called the \emph{inner structure superalgebra of the $L$-supermodule $M$}.
\end{df}


The following result corresponds to the super case of part of \cite[Lemma~1.1]{F73}. Note that the fact that $b$ is well-defined is not trivial, and according to Faulkner's proof (where the details are omitted), this follows from the (super)symmetry of $\langle\cdot,\cdot\rangle$. The proof of this detail is omitted in the proof of \cite[Proposition~25]{MFMR09}, where $b$ was assumed to be well-defined without mention to the (super)symmetry of $\langle\cdot,\cdot\rangle$.

\begin{proposition} \label{restrictionTheorem}
Let $(\cV, \langle\cdot,\cdot\rangle)$ be an object in $\GJSP$ and set $L = \instr(\cV)$.
Then $M := \cV^+$ and $M^* := \cV^-$ are faithful dual $L$-supermodules and
\begin{equation}\begin{split}
\Theta \colon M^* \otimes M & \longrightarrow L \\
f \otimes v & \longmapsto \nu(f,v) := (D_{f,v}, - \eta_{f,v} D_{v,f})
\end{split}\end{equation}
is an epimorphism of $L$-supermodules. Consequently, $(M^* \otimes M)/\ker\Theta \cong L$.

Furthermore, $\langle\cdot,\cdot\rangle$ induces a nondegenerate homogeneous invariant supersymmetric bilinear form $b \colon L \times L \to \FF$ given by
\begin{equation} \label{bilinearRestriction}
b(\nu(f,v), \nu(g,w)) := \langle \nu(f,v) \cdot g, w \rangle = \langle D_{f,v} g, w \rangle.
\end{equation}
\end{proposition}
\begin{proof}
Since $\langle\cdot,\cdot\rangle$ is superinvariant and nondegenerate, it is clear that $M$ and $M^*$ are dual $L$-supermodules (where $M^* = M^\gets$), and both are faithful because $L \subseteq \End(M^* \oplus M)$ consists of endomorphisms. By the universal property of the tensor product, the bilinear map
$$M^* \times M \longrightarrow L, \quad (f, v) \longmapsto \nu(f,v)$$ corresponds to a unique linear map $M^* \otimes M \to L$, namely the map $\Theta$. It is clear that $\Theta$ is surjective, and for any homogeneous elements $f,g\in M^*$, $v,w\in M$, we have that
\begin{align*}
\Theta\big( & \nu(f,v) \cdot (g \otimes w)\big)
= \Theta\big( (\nu(f,v) \cdot g) \otimes w
+ \eta_{g, D_{f,v}} g \otimes (\nu(f,v) \cdot w)\big) \\
&= \nu(\nu(f,v) \cdot g, w) + \eta_{g, D_{f,v}} \nu(g, \nu(f,v) \cdot w) \\
&= [\nu(f,v), \nu(g,w)] = \nu(f,v) \cdot \nu(g,w)
= \nu(f,v) \cdot \Theta(g \otimes w),
\end{align*}
thus $\Theta$ is an epimorphism of $L$-supermodules.

We will show that $b$ is well-defined. Set $\cM = M^* \otimes M$ and consider the linear map
\begin{equation}\begin{split}
\Lambda \colon \cM \otimes \cM &\longrightarrow \FF, \\
(f\otimes v)\otimes(g\otimes w) &\longmapsto \langle \nu(f,v) \cdot g, w \rangle
= \langle D_{f,v}g, w \rangle.
\end{split}\end{equation}
Note that
$$\ker(\Theta \otimes \Theta) = \ker(\Theta) \otimes \cM + \cM \otimes \ker(\Theta)$$
and
$$\im(\Theta \otimes \Theta) = \im(\Theta) \otimes \im(\Theta) = L \otimes L .$$
We claim that $\Lambda$ restricts to the quotient
$$ (\cM \otimes \cM) / \ker(\Theta \otimes \Theta) \cong L \otimes L .$$
To show the claim we need to prove that $\ker(\Theta \otimes \Theta) \subseteq \ker\Lambda$.
It is clear that $\ker(\Theta) \otimes \cM \subseteq \ker\Lambda$. Since $\langle\cdot,\cdot\rangle$ is supersymmetric, we have that 
\begin{equation*}
\Lambda\big((f\otimes v)\otimes(g\otimes w)\big) = \langle D_{f,v} g, w \rangle
= \eta_{D_{f,v}, D_{g,w}} \langle D_{g,w} f, v \rangle,
\end{equation*}
which shows that $\cM \otimes \ker(\Theta) \subseteq \ker\Lambda$. Consequently, $\ker(\Theta \otimes \Theta) \subseteq \ker\Lambda$ and $\Lambda$ induces a linear map $\widetilde{\Lambda} \colon L \otimes L \to \FF$. It is clear that $\widetilde{\Lambda}$ corresponds to a bilinear map $L \times L \to \FF$, namely $b$. We have proven that $b$ is well-defined. Note that since $\langle\cdot,\cdot\rangle$ is supersymmetric, we have that
\begin{align*}
b(\nu(f,v), \nu(g,w)) &= \langle D_{f,v}g, w \rangle
= \eta_{D_{f,v}, D_{g,w}} \langle D_{g,w}f, v \rangle \\
&= \eta_{D_{f,v}, D_{g,w}} b(\nu(g,w), \nu(f,v)),
\end{align*}
so that $b$ is supersymmetric. It is clear that $b$ is nondegenerate because $\langle\cdot,\cdot\rangle$ is nondegenerate. Finally, for each homogeneous $x\in L$, $f,g\in M^*$, $v,w\in M$, we have that
\begin{align*}
b([&\nu(f,v),x], \nu(g,w)) = - \eta_{x, D_{f,v}} b([x,\nu(f,v)], \nu(g,w)) \\
&= - \eta_{x, D_{f,v}} b( \nu(x \cdot f, v)
 + \eta_{x,f} \nu(f, x \cdot v), \nu(g, w) ) =_{\eqref{bilinearRestriction}} \\
&= - \eta_{x, D_{f,v}} \langle \nu(x \cdot f, v) \cdot g, w \rangle
 - \eta_{x,v} \langle \nu(f, x \cdot v) \cdot g, w \rangle
=_\eqref{symmetricForm} \\
&= - \eta_{x, D_{f,v}} \eta_{D_{x\cdot f,v}, D_{g,w}} \langle \nu(g,w) \cdot (x \cdot f), v \rangle \\
&\quad - \eta_{x,v} \eta_{D_{f, x\cdot v}, D_{g,w}} \langle \nu(g,w) \cdot f, x \cdot v \rangle \\
&= - (\eta_{x,D_{f,v}}\eta_{D_{x\cdot f, v}, D_{g,w}}) (\eta_{x\cdot f, D_{g,w}}\eta_{x,f})
\langle f, x \cdot (\nu(g,w) \cdot v) \rangle \\
& \quad + (\eta_{x,v} \eta_{D_{f,x\cdot v}, D_{g,w}}) \eta_{f, D_{g,w}}
\langle f, \nu(g,w) \cdot (x \cdot v) \rangle \\
&= - (\eta_{x,v}\eta_{g,v}\eta_{w,v})
\langle f, x \cdot (\nu(g,w) \cdot v) - \eta_{x, D_{g,w}} \nu(g,w) \cdot (x \cdot v)) \rangle \\
&= - (\eta_{x,v}\eta_{g,v}\eta_{w,v})
\langle f, [x,\nu(g,w)] \cdot v \rangle
=_{\eqref{dualityAction}} \\
&= (\eta_{x,v}\eta_{g,v}\eta_{w,v})
(\eta_{f,x}\eta_{f,g}\eta_{f,w})
\langle [x,\nu(g,w)] \cdot f, v \rangle =_{\eqref{bilinearRestriction}} \\
&= (\eta_{x,v}\eta_{g,v}\eta_{w,v})
(\eta_{f,x}\eta_{f,g}\eta_{f,w})
b([x,\nu(g,w)], \nu(f,v)) =_{\eqref{supersymmetric_b}} \\
&= b(\nu(f,v), [x, \nu(g,w)]),
\end{align*}
which shows that $b$ is also invariant.
\end{proof}


The following result generalizes \cite[Lemma~1.1]{F73}:

\begin{theorem}{\textnormal{[\textbf{Faulkner correspondence}]}} \label{correspondenceTheorem} \\
There is a bijective correspondence between (isomorphy classes of) objects in $\FLSM$ and objects in $\GJSP$, which in turn restricts to a correspondence between objects in $\FLM$ and objects in $\GJP$. Furthermore, given $(L, M, b) \in \FLSM$ and its associated object $(\cV, \langle\cdot,\cdot\rangle) \in \GJSP$, we have $\bAut(L, M, b) \simeq \bAut(\cV, \langle\cdot,\cdot\rangle)$.
\end{theorem}
\begin{proof}
The correspondence of objects in both classes is consequence of Propositions~\ref{GJP}, \ref{instrTheorem} and \ref{restrictionTheorem} (it is easy to see that both constructions are inverses of each other). Let $(L, M, b)$ and $(\cV, \langle\cdot,\cdot\rangle)$ be associated objects in this correspondence.

\smallskip

Fix $(\varphi^-, \varphi^+) \in \Aut(\cV, \langle\cdot,\cdot\rangle)$. Note that $\varphi^- = ((\varphi^+)^\gets)^{-1}$. Consider the map
$\phi_0 \colon \End(\cV^-) \times \End(\cV^+) \to \End(\cV^-) \times \End(\cV^+)$ given by $\phi_0(x) := \varphi \circ x \circ \varphi^{-1}$. Let $f \in M^* = \cV^-$ and $v \in M = \cV^+$. Then
\begin{align*}
\phi_0(\nu(f,v))
&= \big(\varphi^- \circ D_{f,v} \circ (\varphi^-)^{-1}, - \eta_{f,v} \varphi^+ \circ D_{v,f} \circ (\varphi^+)^{-1}\big) \\
&= \big(D_{\varphi^-(f), \varphi^+(v)}, - \eta_{f,v} D_{\varphi^+(v), \varphi^-(f)}\big)
= \nu(\varphi^-(f), \varphi^+(v)),
\end{align*}
so that $\phi_0$ restricts to an element $\varphi_0 \in \GL(L)$. Besides, it is easy to see that $\varphi_0 \in \Aut(L)$. Also,
\begin{align*}
b\big(\varphi_0(&\nu(f, v)) , \varphi_0(\nu(g, w))\big)
= b\big(\nu(\varphi^-(f), \varphi^+(v)) , \nu(\varphi^-(g), \varphi^+(w))\big)
=_{\eqref{bilinearRestriction}} \\
&= \langle \nu(\varphi^-(f), \varphi^+(v)) \cdot \varphi^-(g), \varphi^+(w) \rangle =_{\eqref{tripleProductsFromAction}} \\
&= \langle \{ \varphi^-(f), \varphi^+(v), \varphi^-(g) \}, \varphi^+(w) \rangle
= \langle \varphi^-(\{f,v,g\}), \varphi^+(w) \rangle \\
&=  \langle \{ f,v,g \}, w \rangle =_{\eqref{tripleProductsFromAction}}
\quad \langle \nu(f,v) \cdot g, w \rangle =_{\eqref{bilinearRestriction}} \\
&= b\big(\nu(f,v), \nu(g,w)\big),
\end{align*}
so we get that $\varphi_0 \in \Aut(L, b)$. Moreover, for $x\in L$ we have that
\begin{align*}
\langle f, \varphi^+(&x \cdot v) \rangle
= \langle (\varphi^-)^{-1}(f), x \cdot v \rangle =_{\eqref{flipDualBilinearForm}} \\
&= \eta_{f, x\cdot v} \langle x \cdot v, (\varphi^-)^{-1}(f) \rangle =_\eqref{bilinearRestriction} \\
&= \eta_{f, x\cdot v} b\big(x, \nu(v, (\varphi^-)^{-1}(f))\big) \\
&= \eta_{f, x\cdot v} b\big(\varphi_0(x), \varphi_0(\nu(v, (\varphi^-)^{-1}(f)))\big) \\
&= \eta_{f, x\cdot v} b\big(\varphi_0(x), \nu(\varphi^+(v), f)\big) =_{\eqref{bilinearRestriction}} \\
&= \eta_{f, x\cdot v} \langle \varphi_0(x) \cdot \varphi^+(v), f \rangle =_{\eqref{flipDualBilinearForm}}\\
&= \langle f, \varphi_0(x) \cdot \varphi^+(v) \rangle
\end{align*}
with $\langle\cdot,\cdot\rangle$ nondegenerate, which implies that
$\varphi^+(x \cdot v) = \varphi_0(x) \cdot \varphi^+(v)$. Thus $(\varphi_0, \varphi^+) \in \Aut(L, M, b)$ and $\Aut(\cV, \langle\cdot,\cdot\rangle) \leq \Aut(L, M, b)$.

\smallskip

Now, take $(\varphi_0, \varphi^+) \in \Aut(L, M, b)$. Set $\varphi^- := ((\varphi^+)^\gets)^{-1}$, so that we have
$\langle \varphi^-(f), \varphi^+(v) \rangle = \langle f, v \rangle$
for any $f \in M^*$, $v\in M$. Then,
\begin{align*}
b\big(x, \varphi_0(&\nu(f,v))\big) = b\big(\varphi_0^{-1}(x), \nu(f,v)\big) =_{\eqref{bilinearFormsEc}} \\
&= \langle \varphi_0^{-1}(x) \cdot f, v \rangle =_{\eqref{dualityAction}} \\
&= - \eta_{x,f} \langle f, \varphi_0^{-1}(x) \cdot v \rangle
= - \eta_{x,f} \langle \varphi^-(f), \varphi^+(\varphi_0^{-1}(x) \cdot v) \rangle \\
&= - \eta_{x,f} \langle \varphi^-(f), x \cdot \varphi^+(v) \rangle =_{\eqref{dualityAction}} \\
&= \langle x \cdot \varphi^-(f), \varphi^+(v) \rangle =_{\eqref{bilinearFormsEc}} \\
&= b\big(x, \nu(\varphi^-(f), \varphi^+(v))\big)
\end{align*}
with $b$ nondegenerate, thus we get again that $\varphi_0(\nu(f,v)) = \nu(\varphi^-(f), \varphi^+(v))$. Also,
\begin{align*}
\langle & \varphi^-(x \cdot f), v \rangle
= \langle x \cdot f, (\varphi^+)^{-1}(v) \rangle =_{\eqref{dualityAction}} \\
&= - \eta_{x,f} \langle f, x \cdot (\varphi^+)^{-1}(v) \rangle 
= - \eta_{x,f} \langle \varphi^-(f), \varphi^+\big(x \cdot (\varphi^+)^{-1}(v)\big) \rangle \\
&= - \eta_{x,f} \langle \varphi^-(f), \varphi_0(x) \cdot v \rangle =_{\eqref{dualityAction}} \\
&= \langle \varphi_0(x) \cdot \varphi^-(f), v \rangle
\end{align*}
with $\langle\cdot,\cdot\rangle$ nondegenerate, which implies that $(\varphi_0, \varphi^-) \in \Aut(L,M^*,b)$. Finally, note that
\begin{align*}
\{ \varphi^-(f), \varphi^+(v), \varphi^-(g) \}
&= \nu(\varphi^-(f), \varphi^+(v)) \cdot \varphi^-(g)
= \varphi_0(\nu(f,v)) \cdot \varphi^-(g) \\
&= \varphi^-(\nu(f,v) \cdot g)
= \varphi^-(\{ f,v,g \}),
\end{align*}
and similarly $\{ \varphi^+(v), \varphi^-(f), \varphi^+(w) \} = \varphi^+(\{ v,f,w \})$,
thus $(\varphi^-, \varphi^+) \in \Aut(\cV, \langle\cdot,\cdot\rangle)$
and $\Aut(L,M,b) \leq \Aut(\cV, \langle\cdot,\cdot\rangle)$.

The result follows since the same arguments can be used with automorphism group schemes, and because both constructions of automorphisms are inverses of each other. (Note that for the automorphism group schemes, the bilinear forms $b$ and $\langle\cdot,\cdot\rangle$ are extended as $R$-bilinear forms, for each corresponding associative commutative unital $\FF$-algebra $R$. In this case, it is easy to see that the duals relative to $\langle\cdot,\cdot\rangle$ used in the proof are still well-defined.)
\end{proof}

\smallskip

\begin{remark}
Note that if $(\cV, \langle\cdot,\cdot\rangle) \in \GJSP$ and $(L, M, b) \in \FLSM$ are associated objects in the correspondence from Theorem~\ref{correspondenceTheorem}, then for any $\lambda \in \FF^\times$ we have also that $(\cV, \lambda \langle\cdot,\cdot\rangle) \in \GJSP$ and $(L, M, \lambda b) \in \FLSM$ are associated objects.
\end{remark}

\medskip

\begin{remark}
If we had $\bAut(\cV) \simeq \bAut(L, M)$ for associated objects in the Faulkner correspondence, then the Transfer Theorems in \cite{EKmon} show that the classifications up to equivalence (or up to isomorphism) of gradings by abelian grups would be the same for both objects. However, in general we just have $\bAut(\cV, t) \simeq \bAut(L, M, b)$ for associated objects, which implies a bijective correspondence for gradings (by abelian groups) that are well-behaved with the bilinear forms in a certain sense. Note that if an object $(\cV, t) \in \GJP$ is given by a simple Jordan pair with its generic trace, it is well known that $\bAut(\cV, t) = \bAut(\cV)$ (see \cite[(16.7)]{L75}).
\end{remark}

\medskip

\begin{remark}
The problem that motivated this work is how to extend ``good'' bilinear forms from a finite-dimensional simple Kantor pair $\cV$ to its associated Kantor-Lie algebra $L = \kan(\cV) = \bigoplus_{i = -2}^2 L_i$ obtained by the Kantor construction. This can be useful to recover Killing forms (up to multiplication by a scalar), or to find another bilinear form with nice properties when the Killing form is degenerate (which might happen if the characteristic of the field is positive). The answer to this problem is given by the Faulkner construction, as follows. First, we use the bilinear form $\langle\cdot,\cdot\rangle$ on $\cV$ to get the associated bilinear form $b$ on $L_0$, from the Faulkner correspondence. Then, using that $L_2$ and $L_{-2}$ are dual $L_0$-modules, we can use $b$ to get a third bilinear form $L_{-2} \times L_2 \to \FF$, from the Faulkner correspondence. Together, these three bilinear forms determine a unique bilinear form that is homogeneous for the $\ZZ$-grading of the Kantor construction, which also has properties analogous to the ones of a Killing form.
\end{remark}

\section{Tensor products of GJSP} \label{sectionTensor}

In this section we will transfer some results of tensor products from supermodules to generalized Jordan superpairs.

\begin{df}
For $i=1,2$, let $(\cV_i, \langle\cdot,\cdot\rangle_i)$ be objects in $\GJSP$ and $(L_i, M_i, b_i)$ their associated objects in $\FLSM$. Consider the Lie superalgebra $L = L_1 \oplus L_2$ with the bilinear form $b = b_1 \perp b_2$. Also, consider the trivial action of $L_i$ on $M_j$ for $i \neq j$. Then $M = M_1 \otimes M_2$ is a (not necessarily faithful) $L$-supermodule. By Prop.~\ref{GJP}, $(L, M, b)$ can be used to construct an object in $\GJSP$, which will be denoted by $\cV_1 \otimes \cV_2$ (or $\cV_1 \overleftarrow{\otimes} \cV_2$) and referred to as the \emph{(left) tensor product} of the generalized Jordan superpairs $\cV_1$ and $\cV_2$. The \emph{right tensor product}, denoted $\cV_1 \overrightarrow{\otimes} \cV_2$, can be defined similarly. Notice that $(\cV_1 \otimes \cV_2)^\sigma = \cV_1^\sigma \otimes \cV_2^\sigma$ for $\sigma = \pm$, and the parity map is given by $\varepsilon(v \otimes w) := \varepsilon(v) + \varepsilon(w)$ for homogeneous elements $v\in \cV^\sigma_1$, $w\in \cV^\sigma_2$. 

Also, note that $\instr(\cV_1 \otimes \cV_2)$ is a quotient of $L = L_1 \oplus L_2$, and in general $\instr(\cV_1 \otimes \cV_2) \neq L$. (For instance, if $\cV_1 = \cV_2$ are $1$-dimensional simple Jordan pairs and $\cV = \cV_1 \otimes \cV_2$, then $\cV$ is also $1$-dimensional, so that $\dim\instr(\cV) \leq 1 = \dim\instr(\cV_i) $, hence $\instr(\cV) \neq \instr(\cV_1) \oplus \instr(\cV_2) $.) Since the tensor product operator is associative for supermodules, it is also associative for objects in $\GJSP$. It is also clear that $\cV_1 \overleftarrow{\otimes} \cV_2 \cong \cV_2 \overrightarrow{\otimes} \cV_1$. Note that $\cV_1 \overleftarrow{\otimes} \cV_2 = \cV_1 \overrightarrow{\otimes} \cV_2$ in $\GJP$.

\smallskip

We define the \emph{direct sum} of objects $(L_i, M_i, b_i)$ in $\FLSM$, for $i = 1,\dots,n$, as $(\bigoplus_i L_i, \bigoplus_i M_i, \perp_i b_i)$. Notice that direct sums in $\FLSM$ correspond to direct sums in $\GJSP$. Also, note that the tensor product of supermodules is not distributive for the sum (since the Lie superalgebra acting is not preserved). Consequently, the tensor product is not distributive for direct sums in $\GJSP$.
\end{df}

\begin{notation}
Let $\cV$ and $\cW$ be generalized Jordan superpairs, and $R$ an associative commutative unital $\FF$-algebra. We will denote the extension of scalars by $\cV_R := \cV \otimes R$. For each $\lambda \in R^\times$ we have an automorphism $c_\lambda = (c_\lambda^-, c_\lambda^+) \in \Aut_R(\cV_R)$ given by
\begin{equation}
\text{ $c_\lambda^\sigma(x) := \lambda^{\sigma 1}x$, \quad for each $x\in\cV^\sigma_R$. }
\end{equation}

We will denote $\bG_m := \bGL_1$. Consider the subgroup scheme
$$\bAut(\cV) \otimes_{\bG_m} \bAut(\cW) \leq \bGL(\cV^- \otimes \cW^-) \times \bGL(\cV^+ \otimes \cW^+)$$
defined by
\begin{equation} \label{tensorAutSchemesPairs}
(\bAut(\cV) \otimes_{\bG_m} \bAut(\cW))(R) := \{ f \otimes g \med
f \in\Aut_R(\cV_R), g\in\Aut_R(\cW_R) \},
\end{equation}
where $f \otimes g := (f^- \otimes g^-, f^+ \otimes g^+)$. Here, we identify $(\cV \otimes \cW)_R$ with $\cV_R \otimes_R \cW_R$, and $f \otimes g$ means $f \otimes_R g$ (i.e., tensors are $R$-bilinear).

Note that the morphism $\bAut(\cV) \times \bAut(\cW) \to \bAut(\cV) \otimes_{\bG_m} \bAut(\cW)$ sending $(f, g) \in \Aut_R(\cV_R) \times \Aut_R(\cW_R)$ to $f \otimes g \in \Aut_R(\cV_R) \otimes_{R^\times} \Aut_R(\cW_R)$ has kernel
$$ \bT_1(R) := \{(c_\lambda, c^{-1}_\lambda) \in \Aut_R(\cV_R) \times \Aut_R(\cW_R) \med
\lambda \in R^\times \},$$
so we have that
\begin{equation}
\bAut(\cV) \otimes_{\bG_m} \bAut(\cW) \simeq (\bAut(\cV) \times \bAut(\cW))/\bT_1.
\end{equation}
Also, note that $\bT_1 \simeq \bG_m$ is a $1$-torus. More in general, given generalized Jordan superpairs $\cV_1, \dots, \cV_n$, we can define in a similar way the tensor product
$$ \bigotimes_{i=1}^n {}_{\bG_m} \bAut(\cV_i), $$
which is a quotient of $\prod_{i=1}^n \bAut(\cV_i)$ by an $(n-1)$-torus given by
$$ \bT_{n-1}(R) := \{(c_{\lambda_1},\dots, c_{\lambda_n}) \in \prod_{i=1}^n \Aut_R((\cV_i)_R)
\med \lambda_i \in R, \; \prod_{i=1}^n \lambda_i = 1 \}.$$
Note that the tensor products $\bAut(\cV) \otimes_{\bG_m} \bAut(\cW)$ are a particular case of central product of group schemes relative to $\bG_m$ (a definition for central product of groups can be found in \cite[Chap.2, p.29]{G80}).
\end{notation}

\medskip

\begin{proposition} \label{tensorProductProperties}
Let $\cV_i$ be nonzero objects in $\GJSP$ for $i = 1,2$ and $\cV = \cV_1 \otimes \cV_2$. Then:
\begin{itemize}
\item[1)] The bilinear form $\langle\cdot,\cdot\rangle$ on $\cV$ is given by the tensor superproduct of the bilinear forms of $\cV_1$ and $\cV_2$, that is,
$$ \langle f_1 \otimes f_2, v_1 \otimes v_2 \rangle =
\eta_{f_2, v_1} \langle f_1, v_1 \rangle \langle f_2, v_2 \rangle .$$
\item[2)] The generators of $\instr(\cV)$ are of the form
$$ \nu(f_1 \otimes f_2, v_1 \otimes v_2)
= \eta_{f_2, v_1} \Big( \langle f_2, v_2 \rangle \nu(f_1, v_1)
+ \langle f_1, v_1 \rangle \nu(f_2, v_2)
\Big). $$
\item[3)] The triple products on $\cV$, for homogeneous elements $x_i, z_i \in \cV_i^\sigma$ $y_i \in \cV_i^{-\sigma}$, are given by
\begin{align*}
\{x_1 &\otimes x_2, y_1 \otimes y_2, z_1 \otimes z_2\} = \\
&= \eta_{x_2,y_1} \Big( \{x_1,y_1,z_1\} \otimes \langle x_2, y_2 \rangle z_2 
+ \eta_{z_1,x_2}\eta_{z_1,y_2} \langle x_1, y_1 \rangle z_1 \otimes \{x_2,y_2,z_2\} \Big).
\end{align*}
\item[4)] $\bAut(\cV_1, \langle\cdot,\cdot\rangle) \otimes_{\bG_m}
\bAut(\cV_2, \langle\cdot,\cdot\rangle) \leq \bAut(\cV, \langle\cdot,\cdot\rangle)$.
\end{itemize}
\end{proposition}
\begin{proof}
1) This follows because duality of supermodules satisfy the property
$(M \otimes N)^* \cong M^* \otimes N^*$, where the bilinear pairing form $(M^* \otimes N^*) \times (M \otimes N) \to \FF$ is the tensor superproduct of the bilinear pairing forms $M^* \times M \to \FF$ and $N^* \times N \to \FF$.

2) Take homogeneous elements $x\in \instr(\cV)$, $f_i\in\cV_i^-$, $v_i\in\cV_i^+$. We claim that
\begin{equation} \label{auxEq1}
\eta_{x,f_1}\eta_{x,v_1} \langle f_1, v_1 \rangle = \langle f_1, v_1 \rangle.
\end{equation}
Indeed, if $\varepsilon(f_1) = \varepsilon(v_1)$ we have that $\eta_{x,f_1}\eta_{x,v_1} = 1$, otherwise we have that $\langle f_1, v_1 \rangle = 0$ because $\langle\cdot,\cdot\rangle$ is homogeneous, and in both cases the claim follows. Then
\begin{align*}
b\big(&x, \nu(f_1 \otimes f_2, v_1 \otimes v_2)\big) =_\eqref{bilinearFormsEc} \\
&= \langle x \cdot (f_1 \otimes f_2), v_1 \otimes v_2 \rangle =_\eqref{leftTensorMods} \\
&= \langle (x \cdot f_1) \otimes f_2
+ \eta_{x, f_1} f_1 \otimes (x \cdot f_2), v_1 \otimes v_2 \rangle \\
&= \eta_{f_2, v_1} \langle x \cdot f_1, v_1 \rangle \langle f_2, v_2 \rangle
+ \eta_{x, f_1} \eta_{x \cdot f_2, v_1} \langle f_1, v_1 \rangle \langle x \cdot f_2, v_2 \rangle =_\eqref{bilinearFormsEc} \\
&= \eta_{f_2, v_1} b\big(x, \nu(f_1, v_1)\big) \langle f_2, v_2 \rangle
+ \eta_{f_2,v_1} \eta_{x,f_1} \eta_{x,v_1} \langle f_1, v_1 \rangle b\big(x, \nu(f_2, v_2)\big) =_\eqref{auxEq1} \\
&= b\big(x, \eta_{f_2, v_1} \langle f_2, v_2 \rangle \nu(f_1, v_1)
+ \eta_{f_2, v_1} \langle f_1, v_1 \rangle \nu(f_2, v_2) \big),
\end{align*}
and since $b$ is nondegenerate, the property follows.

3) The triple product for homogeneous elements is given by
\begin{align*}
\{&x_1 \otimes x_2, y_1 \otimes y_2, z_1 \otimes z_2\} =_\eqref{tripleProductsFromAction} \\
&= \nu(x_1 \otimes x_2, y_1 \otimes y_2) \cdot (z_1 \otimes z_2) \\
&= \eta_{x_2,y_1} \Big(\langle x_2, y_2 \rangle \nu(x_1, y_1) + 
\langle x_1, y_1 \rangle \nu(x_2, y_2)\Big) \cdot (z_1 \otimes z_2) \\
&= \eta_{x_2,y_1} \Big( \langle x_2, y_2 \rangle (\nu(x_1, y_1) \cdot z_1) \otimes z_2 \\
& \qquad\qquad + \eta_{z_1,x_2}\eta_{z_1,y_2}
\langle x_1, y_1 \rangle z_1 \otimes (\nu(x_2,y_2) \cdot z_2) \Big) =_\eqref{tripleProductsFromAction} \\
&= \eta_{x_2,y_1} \Big( \{x_1,y_1,z_1\} \otimes \langle x_2, y_2 \rangle z_2
+ \eta_{z_1,x_2}\eta_{z_1,y_2} \langle x_1, y_1 \rangle z_1 \otimes \{x_2,y_2,z_2\} \Big).
\end{align*}

4) Let $\varphi_i \in \Aut_R((\cV_i)_R, \langle\cdot,\cdot\rangle)$ for $i = 1,2$ and $\varphi = \varphi_1 \otimes \varphi_2$ (where $\otimes = \otimes_R$). We will first show that $\langle\cdot,\cdot\rangle$ is $\varphi$-invariant. Indeed, for any homogeneous $f_i \in \cV_i^-$, $v_i \in \cV_i^+$, we have that
\begin{align*}
\langle &\varphi^-(f_1 \otimes f_2), \varphi^+(v_1 \otimes v_2) \rangle =
\langle \varphi^-_1(f_1) \otimes \varphi^-_2(f_2), \varphi^+_1(v_1) \otimes \varphi^+_2(v_2) \rangle \\
&= \eta_{f_2,v_1} \langle \varphi^-_1(f_1), \varphi^+_1(v_1) \rangle
\langle \varphi^-_2(f_2), \varphi^+_2(v_2) \rangle
= \eta_{f_2,v_1} \langle f_1, v_1 \rangle \langle f_2, v_2 \rangle \\
&= \langle f_1 \otimes f_2, v_1 \otimes v_2 \rangle,
\end{align*}
which proves that $\langle\cdot,\cdot\rangle$ is $\varphi$-invariant. On the other hand,
\begin{align*}
\{ & \varphi^-(f_1 \otimes f_2), \varphi^+(v_1 \otimes v_2), \varphi^-(g_1 \otimes g_2) \} \\
&= \{ \varphi^-_1(f_1) \otimes \varphi^-_2(f_2),
\varphi^+_1(v_1) \otimes \varphi^+_2(v_2), \varphi^-_1(g_1) \otimes \varphi^-_2(g_2) \} \\
&= \eta_{f_2,v_1} \{ \varphi^-_1(f_1), \varphi^+_1(v_1), \varphi^-_1(g_1) \}
\otimes \langle \varphi^-_2(f_2), \varphi^+_2(v_2) \rangle \varphi^-_2(g_2) \\
&\quad + \eta_{f_2,v_1} \eta_{g_1,f_2}\eta_{g_1,v_2}
\langle \varphi^-_1(f_1), \varphi^+_1(v_1) \rangle \varphi^-_1(g_1)
\otimes \{ \varphi^-_2(f_2), \varphi^+_2(v_2), \varphi^-_2(g_2) \} \\
&= \eta_{f_2,v_1} \varphi^-_1(\{ f_1, v_1, g_1 \}) \otimes \varphi^-_2(\langle f_2, v_2 \rangle g_2) \\
& \quad + \eta_{f_2,v_1} \eta_{g_1,f_2}\eta_{g_1,v_2}
\varphi^-_1(\langle f_1, v_1 \rangle g_1) \otimes \varphi^-_2(\{ f_2, v_2, g_2 \}) \\
&= (\varphi^-_1 \otimes \varphi^-_2) \Big(
\eta_{f_2,v_1} \{ f_1, v_1, g_1 \} \otimes \langle f_2, v_2 \rangle g_2 \\
& \qquad\qquad\qquad + \eta_{f_2,v_1} \eta_{g_1,f_2}\eta_{g_1,v_2}
\langle f_1, v_1 \rangle g_1 \otimes \{ f_2, v_2, g_2 \} \Big) \\
&= \varphi^-(\{ f_1 \otimes f_2, v_1 \otimes v_2, g_1 \otimes g_2 \}),
\end{align*}
and similarly we get 
$$ \{ \varphi^+(v_1 \otimes v_2), \varphi^-(f_1 \otimes f_2), \varphi^+(w_1 \otimes w_2) \}
= \varphi^+(\{ v_1 \otimes v_2, f_1 \otimes f_2, w_1 \otimes w_2 \}). $$
We have proven that $\varphi\in\Aut_R(\cV_R,\langle\cdot,\cdot\rangle)$.
\end{proof}

\begin{remark}
We can define a tensor product for automorphism group schemes in $\FLSM$, analogous to \eqref{tensorAutSchemesPairs}, given by
\begin{equation} \begin{split}
\big( \bAut(L_1,M_1) \otimes_{\bG_m} \bAut(L_2,M_2) \big) (R) :=
\{ & (\varphi_1 \times \varphi_2, \varphi_1^+ \otimes \varphi_2^+) \med \\
& (\varphi_i, \varphi_i^+) \in\Aut_R \big( (L_i)_R, (M_i)_R \big) \}.
\end{split} \end{equation}
By Th.~\ref{correspondenceTheorem}, property Prop.~\ref{tensorProductProperties}-4) is equivalent to
\begin{equation}
\bAut(L_1,M_1,b_1) \otimes_{\bG_m} \bAut(L_2,M_2,b_2) \leq \bAut(L,M,b),
\end{equation}
in $\FLSM$, where $M = M_1 \otimes M_2$, $b = b_1 \perp b_2$, and $L$ is a quotient of $L_1 \oplus L_2$.
\end{remark}

\bigskip

\begin{proposition}
Assume that the base field is algebraically closed. Let $\cV$ be an object in $\GJSP$ and $L = \instr(\cV)$. Assume that $\cV^\sigma$ is an irreducible $L$-supermodule for $\sigma = \pm$, and also that $L = L_1 \oplus L_2$ where $L_1$ and $L_2$ are $b$-orthogonal graded ideals of $L$. Then there are objects $\cV_i$ in $\GJSP$ such that $\cV \cong \cV_1 \otimes \cV_2$, $\instr(\cV_i) \cong L_i$, and $\cV_i^\sigma$ is an irreducible $L_i$-supermodule for $i = 1,2$ and $\sigma = \pm$.
\end{proposition}
\begin{proof}
By Faulkner's correspondence, it suffices to prove that if $V$ is a finite-dimensional irreducible (super)module for some finite-dimensional Lie (super)algebra $L = L_1 \oplus L_2$, where $L_1$ and $L_2$ are graded ideals of $L$, then $V \cong M_1 \otimes M_2$ where $M_i$ is an $L_i$-(super)module and where the action of $L_i$ on $M_j$ is trivial for $i\neq j$. Actually, this result is well-known at least for the non-super case, but in absence of a good reference, a proof will be given. (The author is grateful to A.~Elduque for providing a sketch of the proof for the non-super case.)

\smallskip

Denote $S = L_1$, $T = L_2$. It is clear that $[S, T] = 0$. Let $W$ be an irreducible $S$-subsupermodule of $V$. Since $V$ is irreducible and finite-dimensional, it decomposes as a direct sum of $n$ $S$-supermodules isomorphic to $W$, for some $n\in \NN$. Since $\FF$ is algebraically closed, it follows by Schur's Lemma that $\undHom_S(W,V)$ is $n$-dimensional, so that $\dim L = \dim(\undHom_S(W,V) \otimes W)$. We claim that $\undHom_S(W,V)$ is a $T$-supermodule with action given by $(t \cdot f)(w) := t \cdot f(w)$ for each $t\in T$, $f\in\undHom_S(W,V)$, $w\in W$. For homogeneous $s\in S$, $t \in T$, $f\in \undHom_S(W,V)$, $w \in W$, we have that
\begin{align*}
(t \cdot f)(s \cdot w) &= t \cdot f(s \cdot w) = \eta_{f,s} t \cdot \big(s \cdot f(w)\big) \\
&= \eta_{f,s} \Big( [t,s] \cdot f(w) + \eta_{t,s} s \cdot \big(t \cdot f(w)\big) \Big) \\
&= \eta_{f,s}\Big(0 +  \eta_{t,s} s \cdot \big( (t \cdot f)(w) \big) \Big) \\
&= \eta_{t \cdot f, s} s \cdot \big( (t \cdot f)(w) \big),
\end{align*}
so that $t \cdot f \in \undHom_S(W,V)$ with $\varepsilon(t \cdot f) = \varepsilon(t) + \varepsilon(f)$. On the other hand, for homogeneous elements $x,y\in T$, $f \in \undHom_S(W,V)$, $w \in W$, we have
\begin{align*}
([x,y] \cdot f)(w) &= [x,y] \cdot f(w) = x \cdot (y \cdot f(w)) - \eta_{x,y} y \cdot (x \cdot f(w)) \\
&= \big( x\cdot(y\cdot f)\big)(w) - \eta_{x,y} \big(y\cdot(x\cdot f)\big)(w) \\
&= \big(x\cdot(y\cdot f) - \eta_{x,y} y\cdot(x\cdot f)\big)(w).
\end{align*}
We have proven that $\undHom_S(W,V)$ is a $T$-supermodule. We can consider the trivial actions $T \cdot W = 0$ and $S \cdot \undHom_S(W,V) = 0$, so that $W$ and $\undHom_S(W,V)$ become $L$-supermodules.

Furthermore, the map $\phi \colon \undHom_S(W,V) \otimes W \to V$, $f \otimes w \mapsto f(w)$, is a homomorphism of $L$-supermodules. By dimensions and since $V$ is irreducible, it follows that $\phi$ is a bijection. We have proven that $V$ and $\undHom_S(W,V) \otimes W$ are isomorphic $L$-supermodules. Note that for each $T$-subsupermodule $N$ of $\undHom_S(W,V)$, we have that $N \otimes W$ is an $L$-subsupermodule of $\undHom_S(W,V) \otimes W$; hence, since $V$ is an irreducible $L$-supermodule, it follows that the $T$-supermodule $\undHom_S(W,V)$ is irreducible too.
\end{proof}

\begin{notation}
Let $(\cV, \langle\cdot,\cdot\rangle)$ be a $1$-dimensional object in $\GJSP$ (i.e., $\dim \cV^+ = \dim \cV^- = 1$), and $(L, M, b)$ the associated object in $\FLSM$. Since $\langle\cdot,\cdot\rangle$ is nondegenerate and homogeneous for the $\ZZ_2$-grading, it follows that $\cV$ is either a pair or an antipair. Thus we can consider its parity, denoted by $\varepsilon(\cV)$ or $\varepsilon(M)$, and set $\eta(\cV) = \eta(M) := (-1)^{\varepsilon(\cV)}$. We will also denote $\eta_a := (-1)^a$ for $a \in \ZZ_2$.

Take $f \in \cV^-$, $v \in \cV^+$ such that $\langle v, f \rangle = 1$, or equivalently, $\eta(\cV) \langle f, v \rangle = 1$ because $\eta_{f,v} = \eta(\cV)$. Take the element $x = \nu(v,f) = - \eta(\cV) \nu(f,v) \in L$. Since any other choice for the pair $(f, v)$ has the form $(\alpha^{-1} f, \alpha v)$ for some $\alpha \in \FF^\times$, it follows that the element $x$ does not depend on the choice. Also, $\varepsilon(x) = \varepsilon(f) + \varepsilon(v) = \bar 0$ if $x \neq 0$. Let $\lambda \in \FF$ be the eigenvalue of the action of $x$ on $\cV^+ = M$ (thus $x$ acts on $\cV^- = M^*$ with multiplier $-\lambda$). Equivalently, $\{v,f,v\} = \lambda v$ and $\{f,v,f\} = \eta(\cV) \lambda f$. Let $G = \FF \times \ZZ_2$. It is clear that the object $\cV$ (resp. $M$) is uniquely determined (up to isomorphism) by $\lambda$ and $\varepsilon(\cV)$, and also by the parameter $\alpha = (\lambda, \varepsilon(\cV)) \in G$, so we can denote $\cV$ (resp. $M$) by $\cV_\alpha$ (resp. $M_\alpha$). Also, note that $\lambda = 0$ iff $x = 0$ iff $L = 0$. 

\smallskip

Given $\alpha = (\lambda, a), \beta = (\mu, b) \in G$, we claim that
\begin{equation}
\cV_\alpha \otimes \cV_\beta \cong \cV_{\alpha + \beta}.
\end{equation}
Let $(f, v)$ and $(f', v')$ be pairs of elements, as above, in $\cV_\alpha$ and $\cV_\beta$, respectively. Then $(\eta_{v,f'} f \otimes f', v\otimes v')$ is a pair of elements of $\cV = \cV_\alpha \otimes \cV_\beta$ with $\langle v \otimes v', \eta_{v',f} f \otimes f' \rangle = \langle v, f \rangle \langle v', f' \rangle$ = 1. It is clear that $\varepsilon(\cV) = \varepsilon(\cV_\alpha) + \varepsilon(\cV_\beta) = a+b$ and $\eta(\cV) = \eta_{a+b} = \eta_a\eta_b = \eta(\cV_\alpha)\eta(\cV_\beta)$. Then we have that
\begin{align*}
\{v \otimes v', &~ \eta_{v',f} f \otimes f', ~ v \otimes v' \} = \\
&= \{ v,f,v \} \otimes \langle v', f' \rangle v' +
\eta_{v,v'} \eta_{v,f'} \langle v,f \rangle v \otimes \{ v',f',v'\} \\
&= (\lambda v) \otimes v'
+ v \otimes (\mu v') = (\lambda + \mu) v \otimes v'.
\end{align*}
Hence, $\cV \cong \cV_\gamma$ with $\gamma = (\lambda + \mu, a+b) = \alpha + \beta$, which proves the claim. On the other hand, is is easy to see that for any object $\cV \in \GJSP$ we have that
\begin{equation}
\cV_0 \otimes \cV \cong \cV \cong \cV \otimes \cV_0.
\end{equation}

\medskip

Let $\cV$ and $\cW$ be objects in $\GJSP$, and $\alpha \in G$. Denote $\cV^{[\alpha]} := \cV \otimes \cV_\alpha$. We will say that $\cV$ is a \emph{tensor-shift by $\alpha$} of $\cW$ (and that $\cV$ and $\cW$ are \emph{tensor-shift related}) if $\cV \cong \cW^{[\alpha]}$. It is clear that $\cV \cong \cV^{[0]}$ and $(\cV^{[\alpha]})^{[\beta]} \cong \cV^{[\alpha + \beta]}$, and it follows that tensor-shift relation is an equivalence relation. Note that tensor-shifting with an element $\alpha = (\lambda, \bar 1)$ produces a bijection between the isomorphic classes of pairs and antipairs (although this correspondence may not restrict to particular subclasses of $\GJSP$, e.g., Jordan pairs do not correspond to anti-Jordan pairs).

\medskip

Given $\alpha = (\lambda, a) \in G$, we can identify $\cV^{[\alpha]}$ with the vector spaces of $\cV$, with degree map $\varepsilon_{[\alpha]}(x) := \varepsilon(x) + a$, with the new bilinear form
\begin{equation}
\langle f, v \rangle_{[\alpha]} := \eta_a \eta_{a,f} \langle f, v \rangle
= \eta_a \eta_{a,v} \langle f, v \rangle,
\end{equation}
and shifted triple products given by
\begin{equation} \label{general-shifted-products} \begin{split}
\{ x, y, z \}^+_{[\alpha]} &:= \eta_{a,y} (\{ x, y, z \}^+ + \lambda \langle x, y \rangle z), \\
\{ x, y, z \}^-_{[\alpha]} &:= \eta_a \eta_{a,y} (\{ x, y, z \}^- + \lambda \langle x, y \rangle z),
\end{split} \end{equation}
where we denote $\eta_{a,y} = (-1)^{a \varepsilon(y)}$. Using \eqref{auxEq1} it follows easily that
\begin{equation}
\cV \otimes \cV_\alpha \cong \cV_\alpha \otimes \cV.
\end{equation}
Furthermore, it is easy to see, using \eqref{general-shifted-products}, that
\begin{equation}
\bAut(\cV^{[\alpha]}, \langle\cdot,\cdot\rangle_{[\alpha]}) \simeq \bAut(\cV, \langle\cdot,\cdot\rangle),
\end{equation}
that is, tensor-shifts preserve automorphism group schemes.

\medskip

In particular, if $\alpha = (\lambda, \bar 0) \equiv \lambda$, we can identify $\cV^{[\lambda]}$ with the vector spaces of $\cV$, with the same bilinear form, and shifted triple products given by
\begin{equation} \label{tensor-shift}
\{ x, y, z \}_{[\lambda]} := \{ x, y, z \} + \lambda \langle x, y \rangle z.
\end{equation}
\end{notation}

\bigskip

\begin{example}
Denote $\cM_{m,n} := \cM_{m,n}(\FF)$. Recall that simple Jordan pairs of type $I_{p,q}$ ($p \leq q$) are given by $\cV^{(I)}_{p,q} := (\cM_{p,q}, \cM_{p,q})$ with triple products $\{x,y,z\} := x y^\Tr z + z y^\Tr x$, and generic trace $t \colon \cM_{p,q}^- \times \cM_{p,q}^+ \to \FF$, $t(x,y) := t(xy^\Tr)$. It is easy to see that $(\cV^{(I)}_{p,q}, t)$ is an object in $\GJP$ (the invariance and symmetry of the generic trace follow easily from the property $t(xy) = t(yx)$ for the trace of matrices $x \in\cM_{n,m}$, $y \in\cM_{m,n}$).

The bilinear form and triple products in $\cV = \cV^{(I)}_{1,p} \otimes \cV^{(I)}_{1,q}$ on the basis $\{ e_i \otimes e_j \}$ (tensor product of canonical bases) are given by 
$$ \langle e_i \otimes e_j, e_k \otimes e_l \rangle = \delta_{ik}\delta_{jl}, $$
and
\begin{align*}
\{ e_i &\otimes e_j, e_k \otimes e_l, e_m \otimes e_n \} = \\
&= \{ e_i, e_k, e_m \} \otimes t(e_j, e_l)e_n + t(e_i, e_k)e_m \otimes \{ e_j, e_l, e_n \} \\
&= (\delta_{ik}e_m + \delta_{km}e_i) \otimes \delta_{jl}e_n
+ \delta_{ik}e_m \otimes (\delta_{jl}e_n + \delta_{ln}e_j) \\
&= 2\delta_{ik}\delta_{jl}e_m \otimes e_n + \delta_{km}\delta_{jl} e_i \otimes e_n
+ \delta_{ik}\delta_{ln} e_m \otimes e_j.
\end{align*}
Hence, the tensor-shift $\cV^{[-2]}$ has the same bilinear form, and triple products given by
\begin{equation*}
\{ e_i \otimes e_j, e_k \otimes e_l, e_m \otimes e_n \}_{[-2]} =
\delta_{km}\delta_{jl} e_i \otimes e_n + \delta_{ik}\delta_{ln} e_m \otimes e_j.
\end{equation*}

On the other hand, note that the generic trace and triple products on $\cV^{(I)}_{p,q}$ are given by
$$ t(E_{ij}, E_{kl}) = t(E_{ij}E_{lk}) = \delta_{ik}\delta_{jl}, $$
and
\begin{align*}
\{ E_{ij}, E_{kl}, E_{mn} \} = E_{ij}E_{lk}E_{mn} + E_{mn}E_{lk}E_{ij}
= \delta_{jl}\delta_{km} E_{in} + \delta_{nl}\delta_{ki} E_{mj}.
\end{align*}

Consequently, the map $e_i \otimes e_j \mapsto E_{ij}$ defines an isomorphism of (generalized) Jordan pairs
\begin{equation}
\cV^{(I)}_{p,q} \cong (\cV^{(I)}_{1,p} \otimes \cV^{(I)}_{1,q})^{[-2]}
= \cV^{(I)}_{1,p} \otimes \cV^{(I)}_{1,q} \otimes \cV_{-2}.
\end{equation}
\end{example}

\medskip

\begin{example}
Now recall that the simple Jordan pair of type $IV_n$ is given by $\cV^{(IV)}_n := (\FF^n, \FF^n)$, with generic trace form $t(x,y) := q(x,y)$, and triple products $\{ x,y,z \} := q(x,y)z + q(z,y)x - q(x,z)y$, where $q \colon \FF^n \times \FF^n \to \FF$ is a nondegenerate symmetric bilinear form. Without loss of generality, we can assume that $q$ is the standard scalar product. Note that the generic trace and the triple products are given on the canonical basis $\{e_i\}$ by
$t(e_i, e_j) = \delta_{ij}$ and $\{ e_i, e_j, e_k \} = \delta_{ij}e_k + \delta_{kj}e_i - \delta_{ik}e_j$.
Then we have that $\cV^{(IV)}_n$ is an object in $\GJP$ (the invariance and symmetry of the generic trace can be checked easily).

The bilinear form and triple products on $\cV = \cV^{(I)}_{1,n} \otimes \cV^{(IV)}_n$ are given by
$$ \langle e_i \otimes e_j, e_k \otimes e_l \rangle = \delta_{ik}\delta_{jl}, $$
and
\begin{align*}
\{ e_i &\otimes e_j, e_k \otimes e_l, e_m \otimes e_n \} = \\
&= \{ e_i, e_k, e_m \} \otimes t(e_j, e_l)e_n + t(e_i, e_k)e_m \otimes \{ e_j, e_l, e_n \} \\
&= (\delta_{ik}e_m + \delta_{km}e_i) \otimes \delta_{jl}e_n
+ \delta_{ik}e_m \otimes (\delta_{jl}e_n + \delta_{ln}e_j - \delta_{jn}e_l) \\
&= 2\delta_{ik}\delta_{jl}e_m \otimes e_n + \delta_{km}\delta_{jl} e_i \otimes e_n
+ \delta_{ik}\delta_{ln} e_m \otimes e_j - \delta_{ik}\delta_{jn} e_m \otimes e_l.
\end{align*}
Therefore, the tensor-shift $\cV^{[-2]}$ has the same bilinear form, and triple products given by
\begin{equation*}
\{ e_i \otimes e_j, e_k \otimes e_l, e_m \otimes e_n \}_{[-2]} =
\delta_{km}\delta_{jl} e_i \otimes e_n
+ \delta_{ik}\delta_{ln} e_m \otimes e_j - \delta_{ik}\delta_{jn} e_m \otimes e_l.
\end{equation*}

On the other hand, let $\cV_{\cM_n}$ be the Kantor pair associated to the structurable algebra $\cM_n := \cM_n(\FF)$, where the involution is the transposition. The triple products on $\cV_{\cM_n}$ are given by $\{ x,y,z \} = x y^\Tr z + z y^\Tr x - z x^\Tr y$. Consider the bilinear form $t \colon \cM_n^- \times \cM_n^+ \to \FF$ given by $t(x, y) := t(x y^\Tr)$. Then $(\cV_{\cM_n}, t)$ is an object in $\GJP$ (the invariance and symmetry of the bilinear trace form follow easily from the properties $t(xy) = t(yx)$ and $t(x^\Tr) = t(x)$ for $x,y \in\cM_n$).

Note that the bilinear form and triple products of $\cV_{\cM_n}$ are also given by
$$ t(E_{ij}, E_{kl}) = t(E_{ij}E_{lk}) = \delta_{ik}\delta_{jl}, $$
and
\begin{align*}
\{ E_{ij}, E_{kl}, E_{mn} \} &= E_{ij}E_{lk}E_{mn} + E_{mn}E_{lk}E_{ij} - E_{mn}E_{ji}E_{kl} \\
&= \delta_{jl}\delta_{km} E_{in} + \delta_{nl}\delta_{ki} E_{mj} - \delta_{nj}\delta_{ik} E_{ml}.
\end{align*}

It follows that the map $e_i \otimes e_j \mapsto E_{ij}$ defines an isomorphism of generalized Jordan pairs
\begin{equation}
\cV_{\cM_n} \cong (\cV^{(I)}_{1,n} \otimes \cV^{(IV)}_n)^{[-2]}
= \cV^{(I)}_{1,n} \otimes \cV^{(IV)}_n \otimes \cV_{-2}.
\end{equation}
\end{example}

\bigskip

\noindent\textbf{Acknowledgements} \\
Thanks are due to A. Elduque for providing very useful advice and references. Also, I am thankful to the anonymous referee for reading the paper and giving some corrections.


\end{document}